\documentclass[a4paper,10pt]{article}

\usepackage{verbatim}
\usepackage{graphicx}
\usepackage{float}

\usepackage{amsmath}
\usepackage{amsthm}
\usepackage{amssymb}
\usepackage{amsbsy}

\usepackage{fouridx}
\usepackage{fullpage}

\usepackage{a4wide}
\usepackage{a4}

\usepackage[bottom]{footmisc}

\numberwithin{equation}{section}
\numberwithin{figure}{section}
\newtheorem{Theorem}{Theorem}
\numberwithin{Theorem}{section}

\newtheorem{Definition}[Theorem]{Definition}

\newtheorem{Corollary}[Theorem]{Corollary}

\usepackage[toc,page]{appendix}

\allowdisplaybreaks

\begin{document}

\title{A comparative study of stochastic resonance for a model with two pathways by escape times, linear response, invariant measures and the conditional Kolmogorov-Smirnov Test}

\author{
Tommy Liu\footnote{University of Reading tommy.liu.academic@gmail.com}
}
\maketitle

\begin{abstract}
We consider stochastic resonance for a diffusion with drift given by a potential, which has two metastable states and two pathways between them. Depending on the direction of the forcing, the height of the two barriers, one for each path, will either oscillate alternating or in synchronisation. 
We consider a simplified model given by a continuous time Markov Chains with two states. 
This was done for  alternating and synchronised wells. 
The invariant measures are derived for both cases and shown to be constant for the synchronised case. 
A PDF for the escape time from an oscillatory potential is studied. 
Methods of detecting stochastic resonance are presented, which are linear response, signal-noise ratio, energy, out-of-phase measures, relative entropy and entropy. 
A new statistical test called the  conditional Kolmogorov-Smirnov test is developed, which can be used to analyse stochastic resonance.  
An explicit two dimensional potential is introduced, the critical point structure derived and the dynamics, the invariant state  and escape time studied numerically. 
The six measures are unable to detect the stochastic resonance in the case of synchronised saddles. The distribution of escape times however not only shows a clear sign of stochastic resonance, but changing the direction of the forcing from alternating to synchronised  saddles an additional resonance at double the forcing frequency starts to appear. 
The conditional KS test reliably detects the stochastic resonance.
This paper is mainly based on the thesis \cite{tommy_thesis}.
\end{abstract}

\tableofcontents

\section{Introduction}

\subsection{Outline of Problem}

Consider the following problem. 
Let $X_t^\epsilon$ be the random variable describing the trajectory of a diffusion process in $\mathbb{R}^r$ where $t$ is the time and $\epsilon^2$ is the variance level. 
More precisely we consider processes described by the following type of stochastic differential equation 
\begin{align*}
dX^\epsilon_t=b\left(X^\epsilon_t,t\right)dt+\epsilon\,dW_t
\end{align*}
where $b:\mathbb{R}^r\times \mathbb{R}\longrightarrow \mathbb{R}^r$ and $W_t$ is a Wiener process in $\mathbb{R}^r$. 
We suppose that the drift term $b$ has the form
\begin{align*}
b(x,t)=-\nabla V_0 (x) + F\cos \Omega t 
\end{align*}
where $F,x\in\mathbb{R}^r$ and $V_0:\mathbb{R}^r\longrightarrow \mathbb{R}$ is called the unperturbed potential. 
We consider unperturbed potentials with two or more minimas (wells). 
Most importantly, we consider potentials where there are multiple pathways between the wells. 
To our knowledge systems with two pathways have not been studied in the context of stochastic resonance.

Consider the case $\Omega=0$ and where the noise $\epsilon$ is very small. 
The particle will stay very close to one of the wells of the potential and will occasionally escape to the other well. 
The time of the actual transition from one well to the other is very short compare to the time it stays in any particular well. 

Now consider the case where $\Omega>0$. For particular choices of $\Omega>0$ and $\epsilon>0$, these transitions between the two wells will become synchronised with the driving frequency $\Omega$. 
This is called stochastic resonance. 
Thus the term \emph{noise induced synchronisation} was used for systems where the amplitude of the forcing $F$ was not large \cite{tk_14, tk_15} (see also the discussions in \cite{tk_16}). 
New insights into the exact manner of these synchronised transitions will be studied in this paper,
which may be more appropriate  in light of the results obtained in this paper.

For small noise $\epsilon$, one would expect that stochastic resonance depends only on the essential properties of the system, such as the height difference between the wells and the pathways for escape. 
We investigate what effects these multiple pathways have on the appearance of stochastic resonance. 
Varying $F$, $\Omega$ and $\epsilon$ should thus reveal the qualitative structure of the unperturbed potential $V_0$. 
In this paper we test this paradigm by studying a two dimensional example with two wells and two independent pathways between them, see Section
\ref{sum_chap_mexican_hat_toy_model}.

\subsection{Historical Background}

Stochastic resonance has attracted  interest among mathematicians and physicist. 
An overview of the studies that have occurred in both physics and mathematics are given here. 

\subsubsection{Physical Background}

Stochastic resonance was first observed in 1981 \cite{benzi81,tk_17c,tk_17b}. 
The first example \cite{benzi81} considered transitions between two metastable states to model the cyclic occurrences of ice ages. 
Since then many examples of stochastic resonance were found in optics 
\cite{opt_1_PhysRevLett.60.2626,
opt_2_Guidoni1995,
opt_3_PhysRevA.49.2199,
opt_4_PhysRevLett.68.3375},
electronics 
\cite{elec_1_FAUVE19835,
elect_2_PhysRevE.49.R1792,
elect_3_Mantegna1995,
elect_4_PhysRevLett.76.563,
elect_5_PhysRevLett.74.3161,
elect_6_:/content/aip/journal/jap/76/10/10.1063/1.358258,
elect_7_PhysRevA.39.4323,
elect_8_PhysRevE.49.4878,
elect_9_PhysRevLett.67.1799},
neuronal systems 
\cite{neuron_1_PhysRevLett.67.656}, 
quantum systems 
\cite{quant_1_:/content/aip/journal/jap/77/6/10.1063/1.358720,quant_2_:/content/aip/journal/apl/66/1/10.1063/1.114161} 
and paddlefish 
\cite{fish_1_PhysRevLett.84.4773,fish_2_FREUND200271}.
Stochastic resonance could be thought of as quasi-deterministically periodic transition between two metastable states. 
For example, the climate of the Earth could be modelled by two states. There is a state corresponding to an Ice Age and another corresponding to the opposite of an Ice Age, a so-called ``Hot Age''. As the Earth's climate cyclically changes many times between Cold Ages and Hot Ages, its behaviour could be modelled by stochastic resonance.

A range of techniques for example
linear response 
\cite{lin_1_TUS:TUS1787,lin_2_doi:10.1137/0143037}, 
signal-to-noise ratio
\cite{sign_1_PhysRevA.39.4668,sign_2_escape_2_PhysRevA.41.4255}
and distribution of escape times
\cite{escape_1_PhysRevA.42.3161,
sign_2_escape_2_PhysRevA.41.4255,
escape_3_PhysRevE.49.4821}
were used to define, analyse and study stochastic resonance. 
These techniques along with other examples of stochastic resonance are reviewed in the long overview paper  by Gammaitoni, H\"anggi, Jung and Marchesoni \cite{RevModPhys.70.223}. 
We will evaluate the usefulness of some of these techniques for our problem, see Section \ref{sum_chap_results}.

\subsubsection{Mathematical Background}

There are various mathematical studies of stochastic resonance. 
These often involve different orders of approximations for small noise levels. 
The first and second order of approximations are discussed below.
Adiabatic large deviation is also presented. 

In the first leading order of approximation, a key element of study is to control the escape times from the wells as given by the so called large deviation theory, see the monograph of Freidlin and Wentzell \cite{freidlin98}.
The distribution of the exit time was derived by Day in \cite{tk_19} and by 
Galves,
Kifer, Olivieri and Vares
\cite{freid_2_doi:10.1137/1119057,
oliveri_expoen,
freid_3}. 
To go beyond leading order has been much more difficult for the transition problem between two wells as WKB theory could up to now not be rigorously applied.

The next order of approximation was rigorously derived by  Bovier, Eckhoff, Gayrard, Klein \cite{Bovier02metastabilityin} and Berglund and Gentz  \cite{kram_2_2008arXiv0807.1681B} using techniques from potential theory.  
Berglund and Gentz in a series of papers studied the situation of low, non-quadratic barriers and drifts not given by autonomous potentials 
\cite{kram_2_2008arXiv0807.1681B, tk_16}. 
A review of different techniques used to derive Kramers' formula can be found in the review paper
\cite{krammer_review_2011arXiv1106.5799B}. 

In \cite{ld_1_comp_1_Freidlin2000333} Friedlin considered stochastic resonance in the adiabatic regime. 
This means the diffusion can effectively be described by a Markov process which describes the jumps between wells. 
This problem was revisited by Hermann, Imkeller and Pavlyukevich, see Chapter 4 in \cite{tran2014} and references therein, to derive results uniformly for varying time scale to identify the optimal resonance point asymptotically for small noise even outside the adiabatic regime leading to different logarithmic corrections including the famous cycling effect discovered by Day \cite{tk_20}, see also \cite{tk_18} for the connection with stochastic resonance.
Escape time outside of adiabatic regime is studied in \cite{tk_00_doi:10.1137/120887965}.

As mentioned above in leading order the transitions of the diffusion process $X^\epsilon_t$ between the wells can be approximated by a two state Markov Chain $Y^\epsilon_t=\pm1$ which have been studied \cite{pav_thesis,pav_imkell,pav03,tran2014}. 
Further comparative studies of the stochastic resonance for the diffusion case $X^\epsilon_t$ versus the Markov Chain $Y^\epsilon_t$ case were done by 
Hermann, Imkeller, Pavlyukevich and Peithmann 
in 
\cite{ld_3_comp_4_imkeller2004stochastic,
comp_2_Herrmann2005,
ld_2_comp_3_herrmann2005large,
herrmann2005}.
A collection of papers on comparative studies between stochastic resonance in  diffusion and Markov Chains can be found in the monograph \cite{tran2014}. 
One of the main conclusions in 
\cite{ld_2_comp_3_herrmann2005large,
herrmann2005,
comp_2_Herrmann2005,
tran2014}
is rigorously showing that using linear response and signal-to-noise ratio to analyse  stochastic resonance in the diffusion case $X^\epsilon_t$ gives a different result to analysing the Markov Chain case $Y^\epsilon_t=\pm1$ with the same techniques even asymptotically in the small noise limit.
Other common methods used to study stochastic resonance  include invariant measures and Fourier transforms. 
We consider six measures of stochastic resonance frequently used and considered by Pavlyukevich in his thesis \cite{pav_thesis,tran2014} which are  linear response, signal-to-noise ratio, energy, out-of-phase measure, relative entropy and entropy.

In this paper we will study stochastic resonance on a two dimensional toy model, in  both the diffusion and Markov Chain cases, and where there are two independent pathways between the wells going through two different saddles. 
The escape times and the six measures of stochastic resonance introduced above are studied.

\section{Static Potential}
We remind ourselves of the theory of escape times and escape rate from a well of a static potential. 
These results follow from large deviation and potential theory. 

\subsection{Large Deviation, Potential Theory and Kramers Formula}
\label{chap_sect_kram}
Let $V:\mathbb{R}^r:\longrightarrow\mathbb{R}$.
Let 
$x\in\mathbb{R}^r$ be a well and  
$z_i\in\mathbb{R}^r$ be  saddles labelled by $i=1\ldots n$. 
The saddles would be gateways providing a passage for escape from the well. 
Define 
\begin{align*}
\Delta V_i=V(z_i)-V(x)
\end{align*}
which is the height difference between the well and the $i$th saddle. 
For small noise $\epsilon$, an approximate expression can be estimated for the escape time of the particle going through  the $i$th saddle. 
It is well known from the theory of large deviation \cite{freidlin98}, that in the lowest order of the noise $\epsilon$ the mean exit time is given by 
\begin{align*}
\tau_i=e^{+2\Delta V_i /\epsilon^2}
\end{align*}
Inverting this gives the escape rate
\begin{align*}
R_i=e^{-2\Delta V_i /\epsilon^2}
\end{align*}
and the total escape rate would be to sum over all the saddles
\begin{align*}
R=\sum_{i=1}^n R_i
=\sum_{i=1}^n e^{-2\Delta V_i /\epsilon^2}.
\end{align*}
The order correction is done by adding a coefficient called Kramers coefficient and the resulting corrected rate is called Kramers rate
\begin{align*}
R_i=k_ie^{-2\Delta V_i /\epsilon^2}
\quad
\text{where} 
\quad
k_i=\frac{\sqrt{\left|\nabla^2V(x)\right|}}{2\pi}
\frac{\left|\lambda(z_i)\right|}{\sqrt{\left\Vert \nabla^2V(z_i) \right\Vert}}
\end{align*}
where $\left|\nabla^2(x)\right|$ denotes the determinant of the  Hessian of the potential at the well $x$, 
$\left\Vert \nabla^2V(z_i) \right\Vert$ denotes the modulus of the determinant of the potential at the saddle $z_i$
and 
$\left|\lambda(z_i)\right|$ denotes the minimum eigenvalue of the Hessian of the potential at the saddle $z_i$. 
This gives the escape rate in the next order of approximation to be 
\begin{align*}
R&=\sum_{i=1}^nR_i
=\sum_{i=1}^n k_ie^{-2\Delta V_i /\epsilon^2}
\end{align*}
which is rewritten as 
\begin{align*}
R=
\frac{\sqrt{\left|\nabla^2V(x)\right|}}{2\pi}
\sum_{i=1}^n
\frac{\left|\lambda(z_i)\right|}{\sqrt{\left\Vert \nabla^2V(z_i) \right\Vert}}
\exp\left\{\frac{-2\left(V(z_i)-V(x)\right)}{\epsilon^2}\right\}
\end{align*}
The last order of approximation for higher noise $\epsilon$ is done by bounding the error on Kramers coefficient. 
This is 
\begin{align*}
k_i=\frac{\sqrt{\left|\nabla^2V(x)\right|}}{2\pi}
\frac{\left|\lambda(z_i)\right|}{\sqrt{\left\Vert \nabla^2V(z_i) \right\Vert}}
\left(
\frac{1}{1+\mathcal{O}\left(\frac{\epsilon^2}{2}\ln\frac{\epsilon^2}{2}\right)}
\right)
\end{align*}
which is rewritten as 
\begin{align}
\label{chap_2_krammers}
\frac{1}{k_i}
=
\frac{2\pi}{\sqrt{\left|\nabla^2V(x)\right|}}
\frac{\sqrt{\left\Vert \nabla^2V(z_i) \right\Vert}}{\left|\lambda(z_i)\right|}
\left[
1+\mathcal{O}\left(\frac{\epsilon^2}{2}\ln\frac{\epsilon^2}{2}\right)
\right]
\end{align}
The derivation of Kramers formula was done in   \cite{Bovier02metastabilityin}.

\section{Oscillatory Potential}
\label{chapter_oscil_times}
Let $V:\mathbb{R}^r\rightarrow\mathbb{R}$ be a potential with two wells. 
This potential is subjected to a periodic forcing $F\in\mathbb{R}^r$ with frequency $\Omega$ and perturbed by noise $\epsilon$, which is described by the SDE.
\begin{align}
\dot{X^\epsilon_t}=-\nabla V+F\cos(\Omega t)+\epsilon \dot{W_t}
\label{chap_4:markov}
\end{align}
where $W_t$ is a Wiener process in $\mathbb{R}^r$ and  $F\in\mathbb{R}^r$. 
We call $X^\epsilon_t$  the diffusion case. 
In this Section we study Equation \ref{chap_4:markov} by considering the escape times between the two wells, modelling Equation \ref{chap_4:markov} by a continuous time Markov Chain, considering six measures of stochastic resonance and a new statistical test called the conditional Kolmogorov-Smirnov Test.

\subsection{Markov Chain Reduction}

Stochastic resonance usually involves studying transitions between two stable states. 
If one solely concentrate on the transitions times one can reduce the model to a continuous time Markov chain $Y_t^\epsilon$ with state space $Y^\epsilon_t=\pm 1$ symbolizing the two stable states.
\begin{align*}
X^\epsilon_t
\longrightarrow
Y^\epsilon_t
\end{align*}
Let $w_{l}(t)$ denote the position of the left well at time $t$ and $w_{r}(t)$ the position of the right well at time $t$. 
Note that $w_l(t)$ and $w_r(t)$ are also continuous in time. 
Let $R\in\mathbb{R}$ be constant. 
The reduction from the  $X^\epsilon_t$ to the Markov Chain $Y^\epsilon_t$ is
\begin{align*}
Y^\epsilon_t&=
\left\{
\begin{array}{lll}
-1 & \text{if} &
\left|X^\epsilon_t-w_l(t)\right|\leq R\\[0.5em]
+1 & \text{if} &
\left|X^\epsilon_t-w_r(t)\right|\leq R\\[0.5em]
Y^\epsilon_s & \text{if} & \left|X^\epsilon_t-w_l(t)\right|> R \quad \text{and} \quad \left|X^\epsilon_t-w_r(t)\right|> R
\end{array}
\right.
\end{align*}
where $s$ is given by 
\begin{align*}
s&=\max\left\{s_1,s_2\right\}\\
\text{where}\quad s_1&=\max_{s<t}\left\{s:\left|X^\epsilon_s-w_l(s)\right|\leq R\right\}\\
\text{where}\quad s_2&=\max_{s<t}\left\{s:\left|X^\epsilon_s-w_r(s)\right|\leq R\right\}
\end{align*}
When $Y^\epsilon_t=-1$ we say the particle is  in the left well 
and when $Y^\epsilon_t=+1$ we say the particle is  in the right well. 
Hence we keep $Y^\epsilon_t$ constant even when the particle is in neither well. 
Only when it enters the other well would $Y^\epsilon_t$ change sign.

The escape time from the left to right well $\tau_{-1+1}$ and from the right to left well $\tau_{+1-1}$ are defined in the following way
\begin{align*}
\tau_{-1+1}&=l\left\{t:Y^\epsilon_t=-1\right\}\quad \text{where} \quad \left\{t:Y^\epsilon_t=-1\right\} \quad \text{is an interval}\\
\tau_{+1-1}&=l\left\{t:Y^\epsilon_t=+1\right\}\quad \text{where} \quad \left\{t:Y^\epsilon_t=+1\right\} \quad \text{is an interval}
\end{align*}
where $l$ denotes the Lebesgue measure. 
In other words the time spent being in the state $Y^\epsilon_t=-1$ is $\tau_{-1+1}$ and the time spent being in the state $Y^\epsilon_t=+1$ is $\tau_{+1-1}$. 
These intervals will always be closed intervals. 
The process $Y^\epsilon_t$ has two states, hence each sample is a piecewise constant function.
The length of each piece is the escape time $\tau_{-1+1}$ or $\tau_{+1-1}$. 
At every point in time it is possible to define a state probability, that is the probability of the trajectory being  $-1$ or $+1$
\begin{align*}
P\left(Y_t=-1\right)=\nu_-(t)
\quad \text{and} \quad 
P\left(Y_t=+1\right)=\nu_+(t)
\end{align*}
A  continuous time Markov Chain model for Equation \ref{chap_4:markov} is studied next by studying its state probabilities.

\subsection{Continuous Time Markov Chain}
Consider a two state continuous time Markov Chain given by $Y_t=\pm1$. 
The probability of transiting from $Y_t=-1$ to $Y_t=+1$ in a small  time interval $[t,t+d t]$ is $p_{-1+1}\left([t,t+d t]\right)\sim p(t)dt$.  
Similarly the probability of transiting from $Y_t=+1$ to $Y_t=-1$ in a small  time interval $[t,t+d t]$ is $p_{+1-1}\left([t,t+d t]\right)\sim q(t)dt$. 
The probability of $Y_t$ staying at $-1$ in the small time interval $[t,t+d t]$ is $p_{-1-1}\left([t,t+d t]\right)$.
Similarly the probability of $Y_t$ staying at $+1$ in the small time interval $[t,t+d t]$ is $p_{+1+1}\left([t,t+d t]\right)$.
These probabilities satisfy 
\begin{align*}
p_{-1-1}+p_{-1+1}=1
\quad \text{and} \quad 
p_{+1-1}+p_{+1+1}=1
\end{align*}
Let $p:\mathbb{R}\rightarrow \mathbb{R}$ and $q:\mathbb{R}\rightarrow \mathbb{R}$ be cyclic functions on the interval $[0,T]$ where $T$ is the period. 
The behaviour of $\nu(t)=
\left(
\begin{array}{cc}
\nu_-(t)&\nu_+(t)
\end{array}
\right)^\dagger$ 
is described by 
\begin{align}
\frac{d\nu}{dt}
&=\label{markov_chain_main_eqn}
Q^\dagger
\nu
\quad 
\text{where}
\quad 
Q=
\left(
\begin{array}{cc}
-p(t) & p(t)\\[0.5em]
q(t) & -q(t)
\end{array}
\right)
\end{align}
and $Q$ is the infinitesimal generator. 
The aim now is to derive the state probability by solving this differential equation  for various forms of $p$ and $q$. 
The extra conditions we use are 
\begin{align}
\nu_-(t)+\nu_+(t)=1\label{markov_chain_main_condition_1}
\quad \text{and} \quad 
\nu'_-(t)+\nu'_+(t)=0
\end{align}
for all times $t$ and the initial conditions at $t=0$ are $\nu_-(0)$ and $\nu_+(0)$. 

After a very long time the state probabilities $\nu_\pm(\cdot)$ should not depend on the initial state probabilities $\nu_\pm(0)$. 
At time infinity $\nu_\pm(\cdot)$ should also be cyclic on $[0,T]$. 
Let the time be given by $t=NT+n$ where $N$ is a discrete number of periods.
This leads us to define the invariant measure as the state probabilities in the limit as $N\longrightarrow\infty$
\begin{align*}
\overline{\nu}(n):=\lim_{N\longrightarrow\infty}\nu(NT+n)
\end{align*}
The rate of convergence to the invariant measure would depend on the value of $p$ and $q$ themselves. Define the relaxation time $T_{relax}$ as the first time $t=T_{relax}$ such that
\begin{align*}
\left|
\overline{\nu}\left(T_{relax}\right)
-\nu\left(T_{relax}\right)
\right|
\leq
e^{-1}
\end{align*}
which is a measure of the rate of convergence to the invariant measure.

\subsubsection{Continuous Time Markov Chain - Alternating Saddles $p\neq q$}
Notice that $p$ may be interpreted as  the probability of escape from the left well and $q$ as  the probability of escape from the right well. 
If $p$ and $q$ are cyclic over $[0,T]$, then this can be interpreted as modelling a potential with periodic forcing in continuous time. 

\begin{Theorem}\label{chap_4_thm:equal_contin}
Let $p\neq q$ and $t\geq0$.
The state probabilities are given by 
\begin{align*}
\nu_-(t)&=\frac{\nu_-(0)+\int_0^tq(s)\exp\left\{\int_0^sp(u)+q(u)\,du\right\}\,ds}{\exp\left\{\int_0^tp(u)+q(u)\,du\right\}}\\[0.5em]
\nu_+(t)&=\frac{\nu_+(0)+\int_0^tp(s)\exp\left\{\int_0^sp(u)+q(u)\,du\right\}\,ds}{\exp\left\{\int_0^tp(u)+q(u)\,du\right\}}
\end{align*}
\end{Theorem}

\begin{Corollary}
\label{corollary_invariant_measure_contin}
For the state probabilities in Theorem \ref{chap_4_thm:equal_contin}
the invariant measures are
\begin{align*}
\overline{\nu}_-(t)
&=
\frac{\int^t_0p(s)g(s)\,ds}{g(t)}
+\frac{\int^T_0p(s)g(s)\,ds}{g(t)\left(g(T)-1\right)}\\[0.5em]
\overline{\nu}_+(t)
&=
\frac{\int^t_0q(s)g(s)\,ds}{g(t)}
+\frac{\int^T_0q(s)g(s)\,ds}{g(t)\left(g(T)-1\right)}\\[0.5em]
\text{where}\quad
g(t)&=\exp\left\{\int^t_0p(u)+q(u)\,du\right\}
\end{align*}
\end{Corollary}

\begin{proof}
We derive the invariant measure for $\overline{\nu}_-(t)$. 
The case for $\overline{\nu}_+(t)$ is similar. 
Consider the fact that $p(\cdot)$ and $q(\cdot)$ are cyclic on $[0,T]$ and let $i$ be an integer, then the following integral can be rewritten as 
\begin{align*}
\int^{(i+1)T}_{iT}
p(s)g(s)\,ds
=
g(T)^i\int^T_0p(s)g(s)\,ds
\end{align*}
Let the time be given by $NT+t$ where $N$ is an integer number of periods. 
This means the following integral can be written as 
\begin{align*}
\int^{NT+t}_0
p(s)g(s)\,ds
=
g(T)^N
\int^t_0
p(s)g(s)\,ds
+
\int^T_0p(s)g(s)\,ds
\sum_{i=0}^{N-1}
g(T)^i
\end{align*}
So the state probability is equal to 
\begin{align*}
\nu_-(NT+t)
=
\frac{\nu_-(0)}{g(T)^Ng(t)}
+
\frac{\int^t_0p(s)g(s)\,ds}{g(t)}
+
\frac{\int^T_0p(s)g(s)\,ds}{g(t)}
\frac{1}{g(T)-1}
\left(
1-\frac{1}{g(T)^N}
\right)
\end{align*}
Letting $N\longrightarrow\infty$ gives the required result. 
\end{proof}

\subsubsection{Continuous Time Markov Chain - Synchronised Saddles $p=q$}
If the forcing is such that the  height of the barrier stays the same for both pathways the same then this corresponds  to the case $p=q$

\begin{Theorem}\label{chap_4_thm:not:equal_contin}
Let $p=q$ and $t\geq0$.
The state probabilities are given by 
\begin{align*}
\nu_-(t)&=\frac{1}{2}-\frac{\nu_+(0)-\nu_-(0)}{2}\exp\left\{-2\int_0^tp(s)\,ds\right\}\\[0.5em]
\nu_+(t)&=\frac{1}{2}+\frac{\nu_+(0)-\nu_-(0)}{2}\exp\left\{-2\int_0^tp(s)\,ds\right\}
\end{align*}
\end{Theorem}

\begin{Corollary}\label{cor_half_continuous}
For the state probabilities in Theorem \ref{chap_4_thm:not:equal_contin}
the invariant measures are
\begin{align*}
\overline{\nu}_-(t)=\overline{\nu}_+(t)=\frac{1}{2}
\end{align*}
\end{Corollary}

\subsection{Probability Density Function of Escape Times}
The  escape rates from the left to right  are denoted by $R_{-1+1}(\cdot)$
and right to left escape rates are denoted by
$R_{+1-1}(\cdot)$. 
The PDFs for the escape times are given by the Theorem below. 

\begin{Theorem}\label{chap_4_pdf_thm}
Let $u$ be the time coordinate of entry into a well, 
then the PDFs for the escape occurring at time coordinate $t>u$ are 
\begin{align*}
p_{-}(t,u)&=R_{-1+1}(t)\exp\left\{-\int^t_uR_{-1+1}(s)\,ds\right\}\\[0.5em] 
p_{+}(t,u)&=R_{+1-1}(t)\exp\left\{-\int^t_uR_{+1-1}(s)\,ds\right\}
\end{align*}
where $p_{-}(t,u)$ is for left to right and $p_{+}(t,u)$ is for right to left. 
\end{Theorem}

\begin{proof}
We consider escaping from the left well.
The right well is similar.
Divide the time interval $[u,t]$ into many small time intervals 
\begin{align*}
\delta t =\frac{t-u}{N}
\end{align*}
Similar to how we derived the invariant measures we want to derive the probability of escape in a very small time interval $[t,t+\delta t]$. 
This is given by 
\begin{align*}
p_{-1+1}([t,t+\delta t])
&=p(t)\delta t \\
&=1-e^{-R_{-1+1}(t)\delta t}\\
&\approx R_{-1+1}(t)\delta t 
\end{align*}
which is valid for small $\delta t$. 
Large deviations allow us to say even more about the escape time $\tau_{-1+1}$ and $\tau_{+1-1}$. 
Theorem 1 in \cite{oliveri_expoen} shows that it is an exponentially distributed random variable. 
The probability of staying in the left well is given by 
\begin{align*}
p_{-1-1}([t,t+\delta t])
&=1-p_{-1+1}([t,t+\delta t])\\
&=1-p(t)\delta t \\
&=1-\left(1-e^{-R_{-1+1}(t)\delta t}\right)\\
&=e^{-R_{-1+1}(t)\delta t}
\end{align*}
We want to know the probability of escaping in the time interval $[t,t+\delta t]$ given that the particle has entered at $u$ and stayed up to time $t$.
This is given by 
\begin{align*}
p_{-1-1}([u,t])p_{-1+1}([t,t+\delta t])
&=\prod_{i=1}^{N}p_{-1-1}\left([u+(i-1)\delta t,u+i\delta t]\right)
p_{-1+1}([t,t+\delta t])\\
&=\prod_{i=1}^{N}\exp\left\{-R_{-1+1}\left(u+(i-1)\delta t\right)\delta t\right\}
p_{-1+1}([t,t+\delta t])\\
&=\exp\left\{\sum_{i=1}^{N}-R_{-1+1}\left(u+(i-1)\delta t\right)\delta t\right\}
p_{-1+1}([t,t+\delta t])\\
&=\exp\left\{-\int^t_uR_{-1+1}(s)\,ds\right\}
R_{-1+1}(t)\delta t
\end{align*}
This completes the proof. 
\end{proof}

\subsubsection{Perfect Phase Approximation of  Probability Density Function of Escape Times}
\label{chap_approx_pdf}
The PDF for the escape times derived in Theorem \ref{chap_4_pdf_thm} had to differentiate between left and right escapes and are conditioned on the time $u$ of entrance into the well. 
Suppose now that $t$ is the escape time from any well, which does not differentiate between left and right escape. 
Note that $t$ is the actual time it takes to escape from a well and is not a time coordinate. 
The PDF for $t$ is given by 
\begin{align*}
p_{tot}(t)=
\frac{1}{2}
\int_{0}^{T}p_-(t+u,u)m_-(u)+p_+(t+u,u)m_+(u)\,du
\end{align*}
This is because  after a long time has elapsed
we would expect that many transitions would have occurred between left and right. 
The number of transitions escaping from the left and right should be roughly the same. 
The $m_-(u)$ is a PDF for the time of entrance into the left well and 
the $m_+(u)$ is a PDF for the time of entrance into the right well. 
We may not have explicit expressions for $m_-(u)$ and $m_+(u)$. 
We derive an approximate expression for $p_{tot}$ without an explicit expressions for $m_-(u)$ and $m_+(u)$. 
Let  $m_-(u)$ and $m_+(u)$ be approximated by 
\begin{align*}
m_-(u)\approx\delta\left(u-T/2\right)
\quad \text{where} \quad 
m_+(u)\approx
\frac{1}{2}\delta\left(u\right)
+
\frac{1}{2}\delta\left(u-T\right)
\end{align*}
where $\delta(\cdot)$ is the Dirac delta function. 
This approximation is used because in the SDEs which we will simulate, 
the times when transition into the left well is greatest is at half the period $u=\frac{T}{2}$
and the times when transition into the right well is greatest is at $u=0$ and $u=T$. 
Due to the fact that $m_-(u)$ and $m_+(u)$ are probabilities a factor of $\frac{1}{2}$ is used in $m_+(u)$. 
Progressing we have 
\begin{align*}
p_{tot}(t)&=
\frac{1}{2}
\int_{0}^{T}p_-(t+u,u)m_-(u)+p_+(t+u,u)m_+(u)\,du\\
&\approx
\frac{1}{2}
\int_{0}^{T}
p_-(t+u,u)\delta\left(u-T/2\right)
+p_+(t+u,u)
\left(
\frac{1}{2}\delta\left(u\right)
+
\frac{1}{2}\delta\left(u-T\right)
\right)
\,du\\[0.5em]
&=
\frac{1}{2}
\left\{
p_-(t+T/2,T/2)
+\frac{1}{2}
p_+(t,0)
+\frac{1}{2}
p_+(t+T,T)
\right\}\\[0.5em]
&=
\frac{1}{2}
\left\{
p_-(t+T/2,T/2)
+\frac{1}{2}
p_+(t,0)
+\frac{1}{2}
p_+(t+0,0)
\right\}\\[0.5em]
&=
\frac{1}{2}
\left\{
p_-(t+T/2,T/2)
+
p_+(t,0)
\right\}\\[0.5em]
&=p_+(t,0)
\end{align*}
This is because for the simulations which we are going to do, the Kramers rate satisfy $R_{-1+1}(t)=R_{+1-1}(t+T/2)$ (see later in Section \ref{chap_mexican_hat_toy_model} and \ref{sum_chap_results} for the geometry of the Mexican Hat Toy Model which justifies this). 
Thus the following approximation 
\begin{align*}
p_{tot}\approx p_+(t,0)
\end{align*}
is only valid for the simulations we do, and not for a general potential. 
We call this way of approximating $m_-(u)$ and $m_+(u)$ the perfect phase approximation.

\subsection{Six Measures of Stochastic Resonance}

\label{chap_analysis_theory}
\label{sum_chap_analysis_theory}

We introduce six possible criteria of measuring how close a process is to exhibiting stochastic resonance.
We call them the six measures denoted by $M_1$, $M_2$, $M_3$, $M_4$, $M_5$ and $M_6$.  
In what follows we will consider so large times,  that the relaxation time has  effectively elapsed for both the diffusion $X^\epsilon_t$  and Markov Chain $Y^\epsilon_t$, in other words
the state probability would have effectively converged to the invariant measure $\overline{\nu}$.
This means that over one period $T=2\pi/\Omega$ of the forcing, the invariant measures will have the properties
\begin{align*}
\overline{\nu}_\pm(t)=\overline{\nu}_\pm(t+T)
\quad \text{and} \quad 
\overline{\nu}_\pm(t)=\overline{\nu}_\mp(t+T/2)
\end{align*}
We obtain the averaged trajectories given by 
\begin{align*}
\left\langle X^\epsilon_t \right\rangle
=E\left(X^\epsilon_t\right)
\quad \text{and} \quad
\left\langle Y^\epsilon_t \right\rangle
=E\left(Y^\epsilon_t\right)
\end{align*}
which are the trajectories obtained after averaging over many realisations. 
Note that $\left\langle X^\epsilon_t \right\rangle$ is calculated 
by averaging over many realisation over many periods. 
The $\left\langle X^\epsilon_t \right\rangle$ is cyclic over one period. 
Notice that $\left\langle Y^\epsilon_t \right\rangle$ is related to the invariant measures by 
\begin{align*}
\left\langle Y^\epsilon_t \right\rangle=\overline{\nu}_+(t)-\overline{\nu}_-(t)
\end{align*}
We introduce the Out-of-Phase Markov Chain defined by 
\begin{align*}
\overline{Y}^\epsilon_t=
\left\{
\begin{array}{l}
0 \quad \text{if} \quad Y^\epsilon_t=-1 \quad \text{and} \quad mod(t,T)\leq T/2\\
1 \quad \text{if} \quad Y^\epsilon_t=-1 \quad \text{and} \quad mod(t,T)> T/2\\
1 \quad \text{if} \quad Y^\epsilon_t=+1 \quad \text{and} \quad mod(t,T)\leq T/2\\
0 \quad \text{if} \quad Y^\epsilon_t=+1 \quad \text{and} \quad mod(t,T)> T/2
\end{array}
\right.
\end{align*}
and similarly the averaged Out-of-Phase Markov Chain is defined by 
\begin{align*}
\left\langle \overline{Y}^\epsilon_t \right\rangle=E\left(\overline{Y}^\epsilon_t\right)
\end{align*}
Define two new functions by 
\begin{align*}
\phi^-(t)=
\left\{
\begin{array}{c}
1 \quad \text{if} \quad mod(t,T)\leq T/2\\
0 \quad \text{if} \quad mod(t,T)> T/2
\end{array}
\right.
\quad \text{and} \quad 
\phi^+(t)=
\left\{
\begin{array}{c}
0 \quad \text{if} \quad mod(t,T)\leq T/2\\
1 \quad \text{if} \quad mod(t,T)> T/2
\end{array}
\right.
\end{align*}
The following trajectories are Fourier transformed
\begin{align*}
\tilde{x}(\omega)&=\mathcal{F}\left(
\langle x_t  \rangle
\right)
=
\langle \mathcal{F} \left(x_t\right) \rangle
\\[0.5em]
\tilde{Y}(\omega)&=\mathcal{F}\left(
\left\langle Y^\epsilon_t \right\rangle
\right)
=
\langle \mathcal{F} \left(x_t\right) \rangle
\end{align*}
The linear response is defined as the intensity of the Fourier Transform at the driving frequency $\Omega$. 
\begin{align*}
X_{lin}=\left|\tilde{x}\left(\frac{\Omega}{2\pi}\right)\right|
\quad \text{and} \quad 
Y_{lin}=\left|\tilde{Y}\left(\frac{\Omega}{2\pi}\right)\right|
\end{align*}
Note that the forcing is periodic and monochromatic. 
Now we can define the six measures.
For the diffusion case only $M_1$ and $M_2$ are defined
\begin{align*}
M_1=\frac{1}{F}X_{lin}
\quad \text{and} \quad 
M_2=\frac{1}{\epsilon F}X_{lin}
\end{align*}
where $F$ is the magnitude of the forcing. 
For the Markov Chain $M_1$, $M_2$, $M_3$, $M_4$, $M_5$ and $M_6$ are all defined as 
\begin{align*}
M_1&=\frac{1}{F}Y_{lin}\\
M_2&=\frac{1}{\epsilon F}Y_{lin}\\
M_3&=\int_0^T \left\langle Y^\epsilon_t \right\rangle^2 dt\\
M_4&=\int_0^T \left\langle \overline{Y}^\epsilon_t \right\rangle dt\\
M_5&=\int_0^T
\phi^-(t)\ln\left(\frac{\phi^-(t)}{\overline{\nu}_-(t)}\right)+
\phi^+(t)\ln\left(\frac{\phi^+(t)}{\overline{\nu}_+(t)}\right)
dt\\
M_6&=\int^T_0
-\overline{\nu}_-(t)\ln\overline{\nu}_-(t)
-\overline{\nu}_+(t)\ln\overline{\nu}_+(t)\,
dt
\end{align*}
Note that in definition of the six measures it is assumed that the process has relaxed to equilibrium. 
We give a few physical interpretation of the six measures 
$M_1$, $M_2$, $M_3$, $M_4$, $M_5$ and $M_6$. 
The $M_1$ is the intensity of the driving frequency $\Omega$ in the spectrum of the Fourier transform. 
The $M_2$ is sometimes called signal-to-noise ratio as it compares this intensity to the noise level $\epsilon$. 
The $M_3$ is sometimes called the energy. 
The $M_4$ is sometimes called the out-of-phase measure since it measures the amount of time the Markov Chain spends in the ``wrong'' well. 
The $M_5$ and $M_6$ are sometimes called relative entropy and entropy respectively, since they measure how far away the invariant measures are from being constant. 
If the invariant measures are constant then these six measures will also be constant. 
Thus it can be understood that these six measures is a measure of how far away the invariant measures are from being constant.

\subsection{Escape Time Statistics}
We will measure the escape time for many consecutive transitions. 
This will result in a collection of measurements of escape times
\begin{align*}
\tau_1,\tau_2,\ldots,\tau_n
\end{align*}
A new method for analysing such a collection of measurements is presented. 
One may be led to think that in the case of synchronized saddle no stochastic resonance is possible. 
To study this in detail we need to have a more careful look at the escape time statistics.  
The problem we are facing is that the distribution of the escape times strictly speaking depends on the entrance time phase $u$. As we have no theoretical result for the distribution of $u$, we are developing here a technique to study the distribution of the escape conditioned on the entrance phase $u$, that is $p(t ,u)$.

\subsubsection{Kolmogorov-Smirnov Test}
First we recall well known results about the Kolmogorov-Smirnov statistic and the Kolmogorov-Smirnov test \cite{ks_test_1}. 
Let
$
\xi_1, \xi_2, \ldots, \xi_n
$
be $n$ independently and identically distributed real random variables,
where each $\xi_i$ is distributed with  CDF $F(\cdot)$. 
The empirical CDF $F_n(\cdot)$ is defined. This gives
\begin{align*}
P(\xi_i\leq x)=F(x)
\quad \text{and} \quad
F_n(x)=
\frac{1}{n}
\sum_{i=1}^n
\mathbf{1}_{(-\infty,x]}
(\xi_i)
\end{align*}
where $\mathbf{1}_{A}$ is the indicator function for a set $A$. 
Consider the following
\begin{align*}
D_n=
\left\Vert
F_n-F
\right\Vert_\infty
=\sup_{x\in\mathbb{R}}
\left|
\frac{1}{n}
\sum_{i=1}^n
\mathbf{1}_{(-\infty,x]}
(\xi_i)
-F(x)
\right|
\end{align*}
where $D_n$ is called the Kolmogorov-Smirnov statistic or KS statistic. 
We define what we mean by the null hypothesis. 

\begin{Definition}
Let $\xi_1,\xi_2,\ldots,\xi_n$ be $n$ real random variables.
The null hypothesis is 
that each $\xi_i$ is independently distributed with CDF $F(x)$.
\end{Definition}

\noindent We want to know how large or small $D_n$ needs to be before deciding whether to reject the null hypothesis. 
The following Theorem offers a remarkable answer to this problem.

\begin{Theorem}\label{chap_4_thm_ks_test}
Suppose the null hypothesis is true,
then the distribution of $D_n$ depends only on $n$. 
\end{Theorem}

\noindent 
This distribution is called the KS distribution.
There is a Theorem which describes the asymptotic behaviour of the KS distribution. 

\begin{Theorem}
In the limit $n \longrightarrow \infty$, $\sqrt{n}D_n$ is asymptotically Kolmogorov distributed with the CDF 
\begin{equation*}
Q(x)=1-2\sum_{k=1}^\infty (-1)^{k-1}e^{-2k^2x^2}
\end{equation*}
that is to say 
\begin{equation*}
\lim_{n\longrightarrow \infty} P(\sqrt{n} D_n \leq x)=Q(x). 
\end{equation*}
\end{Theorem}

\subsubsection{Conditional Kolmogorov-Smirnov Test}
This conditional Kolmogorov-Smirnov Test is developed in the paper\cite{tommy_paper_A} and the thesis \cite{tommy_thesis}. 
Let $\zeta_1,\zeta_2,\ldots,\zeta_n$ be $n$ iid real random variables. 
They are $n$ empirical observations of a random variable $\zeta$. 
Now suppose that each of the $\xi_1,\xi_2,\ldots,\xi_n$ is conditioned and dependent on the corresponding $\zeta_1,\zeta_2,\ldots,\zeta_n$.
The conditional CDF $F(\cdot,\cdot)$ is
\begin{align*}
P(\xi_i\leq x\,|\,\zeta_i)=F_{\zeta_i}(x)=\int^x_{-\infty}f(s,\zeta_i)\,ds.
\end{align*}
But $\xi_1,\xi_2,\ldots,\xi_n$ are empirical measurements of the same random variable $\xi$. 
The CDF for $\xi$ is 
\begin{align*}
P(\xi\leq x)=F(x)=\int^x_{-\infty}\int^{u=+\infty}_{u=-\infty}f(s,u)m(u)\,du\,ds
\end{align*}
where $m(\cdot)$ is the PDF for $\zeta$.
In our context we have the problem that the random variables are not identically distributed under the null hypothesis. 
The  $\xi_1,\xi_2,\ldots,\xi_n$ and $\zeta_1,\zeta_2,\ldots,\zeta_n$ are obtained experimentally and $F_{\zeta_i}(\xi_i)$ can be  calculated but a PDF for $\zeta_i$, that is $m(\cdot)$, has no easy expression. 
We still want to perform a statistics test that is similar to the KS test even in such situations where the distribution $m(\cdot)$ of $\zeta$ is unknown. 
First we define what we call the total null hypothesis and the conditional null hypothesis. 

\begin{Definition}
Let $\xi_1, \xi_2, \ldots, \xi_n$ be $n$ empirical observations of a random variable $\xi$. 
The total null hypothesis is that
$\xi$ is distributed with the CDF $F(\cdot)$.
The conditional null is that
 each $\xi_i$ is distributed with the conditional CDF $F_{\zeta_i}(\cdot)$.
\end{Definition}

\noindent A new statistical test is developed, which is similar to the KS test. 

\begin{Theorem}
Suppose the conditional null hypothesis is true. 
Let $F_{\zeta_i}(\cdot)$ be continuous. 
Let $S_n$ be the statistic given by 
\begin{align*}
S_n=\sup_{x\in[0,1]}
\left|
\frac{1}{n}
\sum^n_{i=1}
\mathbf{1}_{[0,x]}
\left(
F_{\zeta_i}(\xi_i)
\right)
-x
\right|
\end{align*}
then $S_n$ is KS distributed. 
\end{Theorem}

\subsection{Adiabatic Large Deviation}
We have to stress that this paper is built on three approximations, which form the backbone of all the research presented. These are small noise approximation, adiabatic approximation and perfect phase approximation. 

Perfect phase approximation only works for small noise. This is because the noise is so small the particle will only escape when the maximum probability to escape has arrived. When the minimum probability to escape is present it will almost never escape. This is the idea behind the perfect phase approximation. 

Notice one subtlety behind all the theory presented in this Section. 
The derivations involved probabilities of escape $p$ and $q$ and the escape rates $R_{-1+1}$ and $R_{+1-1}$.
But it was assumed that $p$, $q$, $R_{-1+1}$ and $R_{+1-1}$ are accurately known no matter how large or small the noise level $\epsilon$ is and no matter how fast or slow the driving frequency $\Omega$ is. 
But such ideal expressions for $p$, $q$, $R_{-1+1}$ and $R_{+1-1}$ are not known. 

When we come to do the analysis in Section \ref{sum_chap_results}, the $p_{tot}$ is calculated with the approximation $p_{tot}\approx p_+(t,0)$. When the rates $R_{-1+1}$ and $R_{+1-1}$ are needed they are calculated using Kramers formula as though it is escape from a static potential in the small noise limit. 
This means an oscillatory potential is being approximated by a static potential which is the adiabatic approximation.

In the paper
\cite{adiabatic_large_deviation_herrmann2006}
the adiabatic approximation was justified in the small noise, slow forcing limit using time dependent large deviation theory, that is, it was shown asymptotically the escape times are given by the adiabatic approximation. This result is only for the leading term, whether the analogue result holds for the Kramers rate is unknown.

\section{Mexican Hat Toy Model}
\label{sum_chap_mexican_hat_toy_model}
\label{chap_mexican_hat_toy_model}

The main object of consideration of this paper, which is called the Mexican Hat Toy Model, is now introduced. 
Let $a>0$, $b>0$ and $V_0:\mathbb{R}^2\longrightarrow \mathbb{R}$ be a real function from the plane to the line.
The unperturbed potential is defined as
\begin{align*}
V_0(x,y)=\frac{1}{4}r^4-\frac{1}{2}r^2-ax^2+by^2
\quad \text{where} \quad r=\sqrt{x^2+y^2}
\end{align*}
Let $F_x, F_y \in  \mathbb{R}$ be the forcing. The potential with forcing $V_F$ is defined as 
\begin{align*}
V_F(x,y)
&=\frac{1}{4}r^4-\frac{1}{2}r^2-ax^2+by^2+F_xx+F_yy\\
&=\frac{1}{4}r^4-\frac{1}{2}r^2-ax^2+by^2+\mathbf{F}\cdot\mathbf{x}\\
&=V_0+\mathbf{F}\cdot\mathbf{x}
\end{align*}
written more compactly in vector notation.
The forcing will clearly have a magnitude and direction given by 
\begin{align*}
F=\sqrt{F_x^2+F_y^2}
\quad \text{and} \quad 
\phi=\tan^{-1}\left(\frac{F_y}{F_x}\right)
\end{align*}
We will study the critical points which are solutions to the simultaneous equations 
\begin{align}
\frac{\partial V_F}{\partial x}=0
\quad \text{and} \quad 
\frac{\partial V_F}{\partial y}=0
\label{solv_equation_hat}
\end{align}
It is easy to show that when there is no forcing $F=0$, there are five critical points, that is two wells, two saddles and one hill. 
Stochastic resonance is studied when the forcing is small enough such that the topology of the potential does not change significantly. 
This is because if the forcing is too large (beyond criticality) then transitions are almost certain, and there is little point to consider stochastic resonance. 

In our context, what we mean by the topology of the potential not changing significantly is when the forcing $F$ is small enough such that there are still five critical points and none of them have changed their nature. 
We have the following Theorem. 

\begin{Theorem}\label{x_crit_thm}
Let $F_x>0$, $F_y=0$ and $b<\frac{1}{2}$.
Let $F_x$ be bounded by $F_x^{sad}$ and $F_x^{crit}$ where 
\begin{align*}
F_x^{sad}=2(a+b)\sqrt{1-2b}
\quad \text{and} \quad 
F_x^{crit}=\sqrt{\frac{4(1+2a)^3}{27}}
\end{align*}
then $V_f$ has five critical points. 
Let  $F_x=0$, $F_y>0$ and $b<\frac{1}{2}$. 
Let $F_y$ be bounded by $F_y^{sad}$ and $F_y^{crit}$ where 
\begin{align*}
F_y^{sad}=2(a+b)\sqrt{1+2a}
\quad \text{and} \quad 
F_y^{crit}=\sqrt{\frac{4(1-2b)^3}{27}}
\end{align*}
then $V_f$ has five critical points. 
\end{Theorem}

\begin{proof}
We prove the case for $F_x>0$, $F_y=0$ and $b<\frac{1}{2}$. The case for $F_x=0$, $F_y>0$ and $b<\frac{1}{2}$ is similar. 
Solving Equation \ref{solv_equation_hat}
leads to 
$
\partial V_F/\partial y
=y(x^2+y^2)-(1-2b)y=0
$
which holds if either $(x^2+y^2)-(1-2b)=0$ or $y=0$. 
The first case is $(x^2+y^2)-(1-2b)=0$ which gives two solutions in $y$.  Having $F_x<F_x^{sad}$ gives two real solutions.
The second case is $y=0$ which would yield a cubic equation with three unknowns. Having $F_x<F_x^{crit}$ gives three real solutions. 
\end{proof}

\noindent The exact properties and behaviour of the critical points are most easily  studied for the cases $F_x\neq0$, $F_0$ and $F_x=0$, $F_y\neq0$. 
For $F_x\neq0$, $F_y\neq0$ a quintic equation with five unknowns is involved.
This means that although an explicit value for the critical forcing when $F_x\neq0$ and $F_y\neq0$ cannot be given analytically, an educated guess can be made 
\begin{align}
F^{crit}=\min\left\{F_x^{sad}, F_x^{crit}, F_y^{sad}, F_y^{crit}\right\}
\label{chap_5_critical_forcing}
\end{align}
that is because a critical force in a general direction must encompass all the other directions. 
More details of the properties of the critical points can be found in the thesis \cite{tommy_thesis}. 
We give an example of how the Mexican Hat Toy Model look like 
\begin{figure}[H]
\centerline{\includegraphics[scale=0.35]{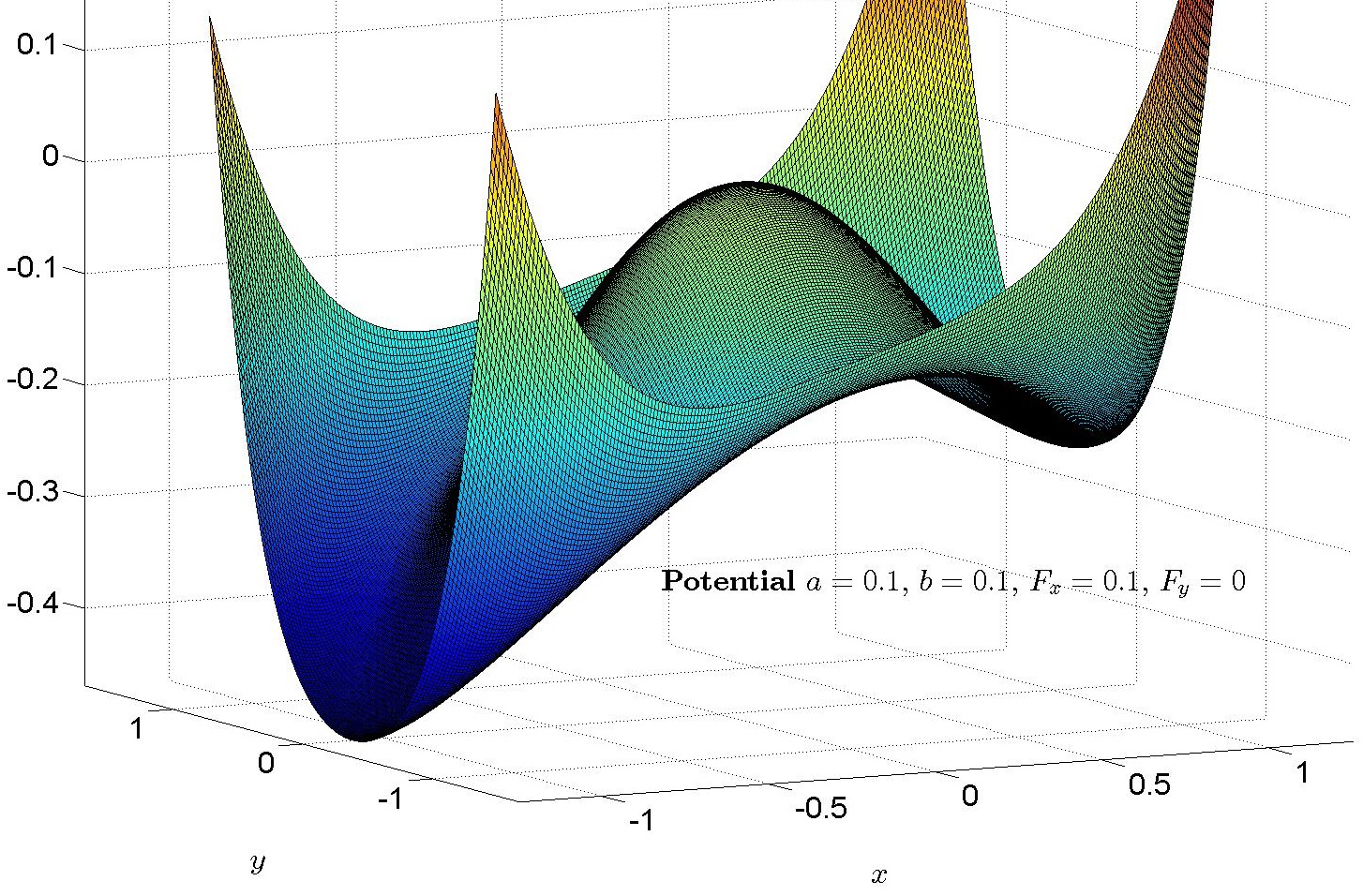}}
\caption{
An example of the potential $V_F(x,y)=\frac{1}{4}r^4-\frac{1}{2}r^2-ax^2+by^2+F_xx+F_yy$ where $r=\sqrt{x^2+y^2}$. 
Here $a=0.1$, $b=0.1$, $F_x=0.1$ and $F_y=0$. 
Notice there are two saddles just ahead of the hill.
The well on the right is higher than the well on the left.
}
\end{figure}

\section{Experimental Method}

\label{sum_chap_results}

We simulate a series of stochastic trajectories for the Mexican Hate Toy Model and analyse them. 
The SDE we want to simulate is 
\begin{align*}
dx&=\left[-\frac{\partial V_0}{\partial x}+F_x\cos \Omega t \ \right]dt+\epsilon \ dw_x\\
dy&=\left[-\frac{\partial V_0}{\partial y}+F_y\cos \Omega t \ \right]dt+\epsilon \ dw_y
\end{align*}
where $V_0:\mathbb{R}^2\longrightarrow\mathbb{R}$ is the unperturbed potential of the Mexican Hat Toy Model, $F_x$ and $F_y$ are the $x$ and $y$ components of the forcing, $\Omega$ is the forcing frequency, $\epsilon$ is the noise level and $w_x$ and $w_y$ are two independent Wiener processes. 
This  SDE can be written alternatively as 
\begin{align*}
dx&=\left[-\frac{\partial V_0}{\partial x}+F\cos\phi\cos \Omega t \ \right]dt+\epsilon \ dw_x\\
dy&=\left[-\frac{\partial V_0}{\partial y}+F\sin\phi\cos \Omega t \ \right]dt+\epsilon \ dw_y
\end{align*}
The Euler method was used to simulate this SDE with the following  parameters being fixed at the following values (see Equation \ref{chap_5_critical_forcing} for $F^{crit}$) 
\begin{align*}
a=0.15 \quad
b=0.1 \quad 
F=0.7F^{crit} \quad 
\Omega=0.001
\end{align*}
The angle of the forcing $\phi$ and the noise level $\epsilon$ were varied. 
The values used were 
\begin{align*}
\epsilon=0.15, 0.16, \ldots, 0.30 
\quad \text{and} \quad 
\phi=0^\circ,75^\circ,78^\circ,81^\circ,84^\circ,87^\circ,90^\circ
\end{align*}
The averaged diffusion trajectories $\langle x_t \rangle$  and $\langle y_t \rangle$ were collected. 
The averaged Markov Chain $\langle Y_t^\epsilon \rangle$ and the averaged Out-of-Phase Markov Chain $\langle \overline{Y}_t^\epsilon \rangle$ were collected as well. 
This would allow for the calculation of the invariant measures $\overline{\nu}_-(\cdot)$ and $\overline{\nu}_+(\cdot)$. 
The time coordinates of the entrance and exit to and from the left and right wells were also collected. 
This would allow for the calculation of the escape times. 
We use the following values for the time step and the radius around the wells.
\begin{align*}
t_{step}=0.014 
\quad \text{and} \quad 
R=0.19
\end{align*}
Note that the period of the forcing is denoted by 
\begin{align*}
T=\frac{2\pi}{\Omega}
\end{align*}
The averaged trajectories were simulated by taking the averaged of 200 realisations. 
Each realisation was 30 periods long, that is a trajectory over the interval $[0,30T]$. 
The initial value of the state probabilities were set at 
\begin{align*}
\nu_-(0)=\nu_+(0)=\frac{1}{2}
\end{align*}
which assists in giving a faster convergence to the invariant measures (see Theorems  \ref{chap_4_thm:equal_contin} and \ref{chap_4_thm:not:equal_contin}). 
We should also stress that a lot of the data and results presented in this Section is just a selection of out a much wider range of results. 
All 112 combinations of the parameters were simulated and analysed. 
Details as to why these range of parameters are chosen for the experiment are given in \cite{tommy_thesis}.

\section{Results}
\label{sum_chap_results}

\subsection{Six Measures Analysis}
\label{conclusion_six_measures}
The six measures are calculated for the diffusion and Markov Chain case for all angles of the forcing $\phi$ and all noise levels $\epsilon$ used in the simulations. 
When $\phi=90^\circ$ the wells were moving up and down but they were always at the same height as each other. 
The distance from either wells to the saddles, which is a gateway for escape, is the same in both wells at all times. 
This means the $\phi=90^\circ$ can be modelled by a synchronised Markov Chain with $p=q$. 
The invariant measures for the $\phi=90^\circ$ case as predicted by Corollary \ref{cor_half_continuous} is $\overline{\nu}_-=\overline{\nu}_+=\frac{1}{2}$, which means
the Fourier Transform of the averaged Markov Chain is predicted to be zero.
This predicts the six measures at $\phi=90^\circ$ to be 
\begin{align*}
M_1=M_2=M_3=0\quad
M_4=\frac{1}{2}T\quad
M_5=M_6=+T\ln(2)
\end{align*}
Note that $\ln(2)=0.6931\approx0.7$. 
Notice that for very low noise level $\epsilon\approx0$ the probabilities of escape from either well is so small it may be approximately modelled by a synchronised Markov Chain with $p\approx q$. 
The results below confirm the predictions for the case of $\phi=90^\circ$. 
Since $M_1$ and $M_2$ differ by a factor, only $M_2$ is shown.

\begin{figure}[H]
\centerline{\includegraphics[scale=0.38]{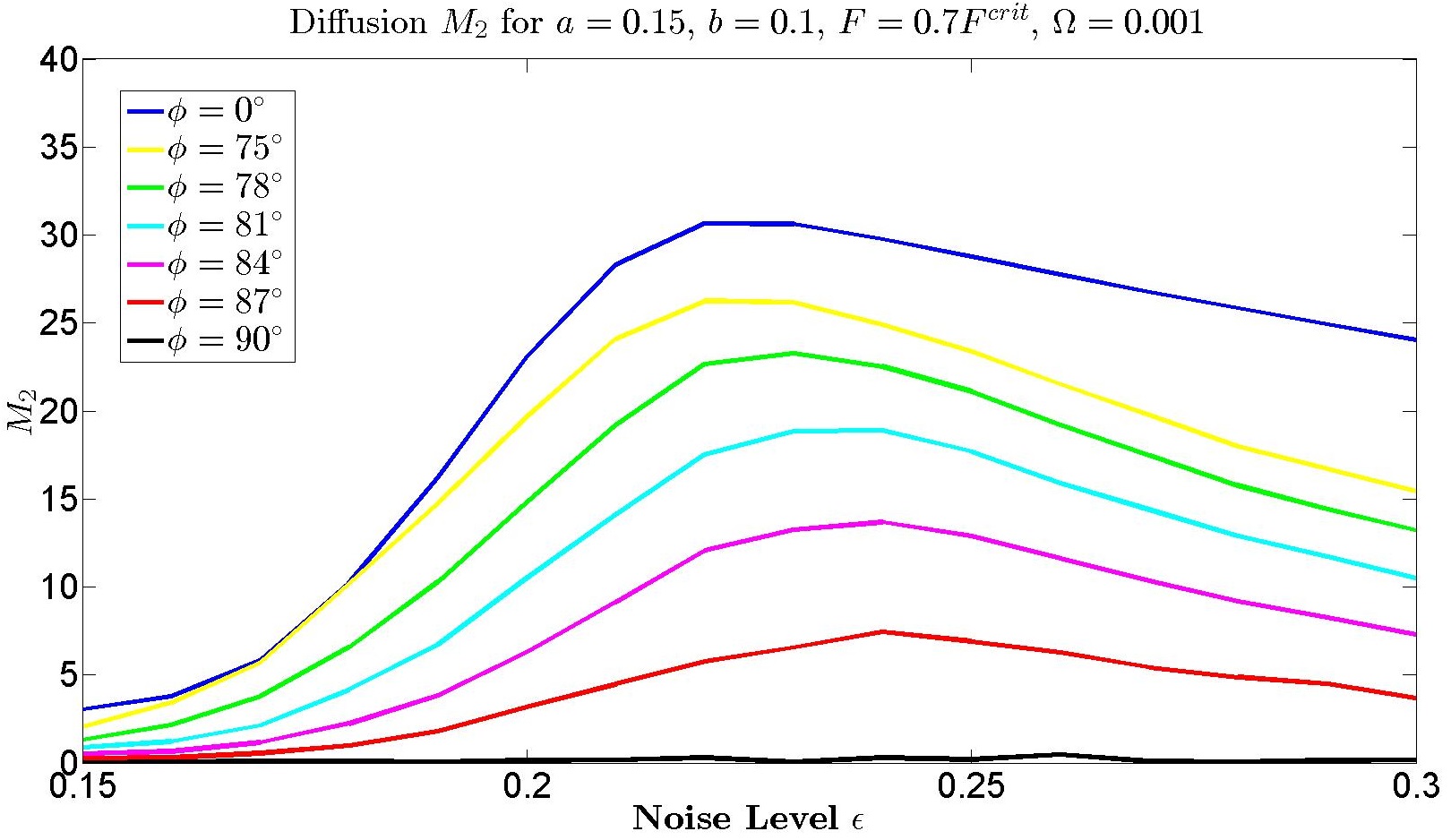}}
\caption{The measure $M_2$ for the diffusion case for various angles and noise levels.}
\label{chap_8_g75_diff_m_2_x_single_measures}
\end{figure}

\begin{figure}[H]
\centerline{\includegraphics[scale=0.38]{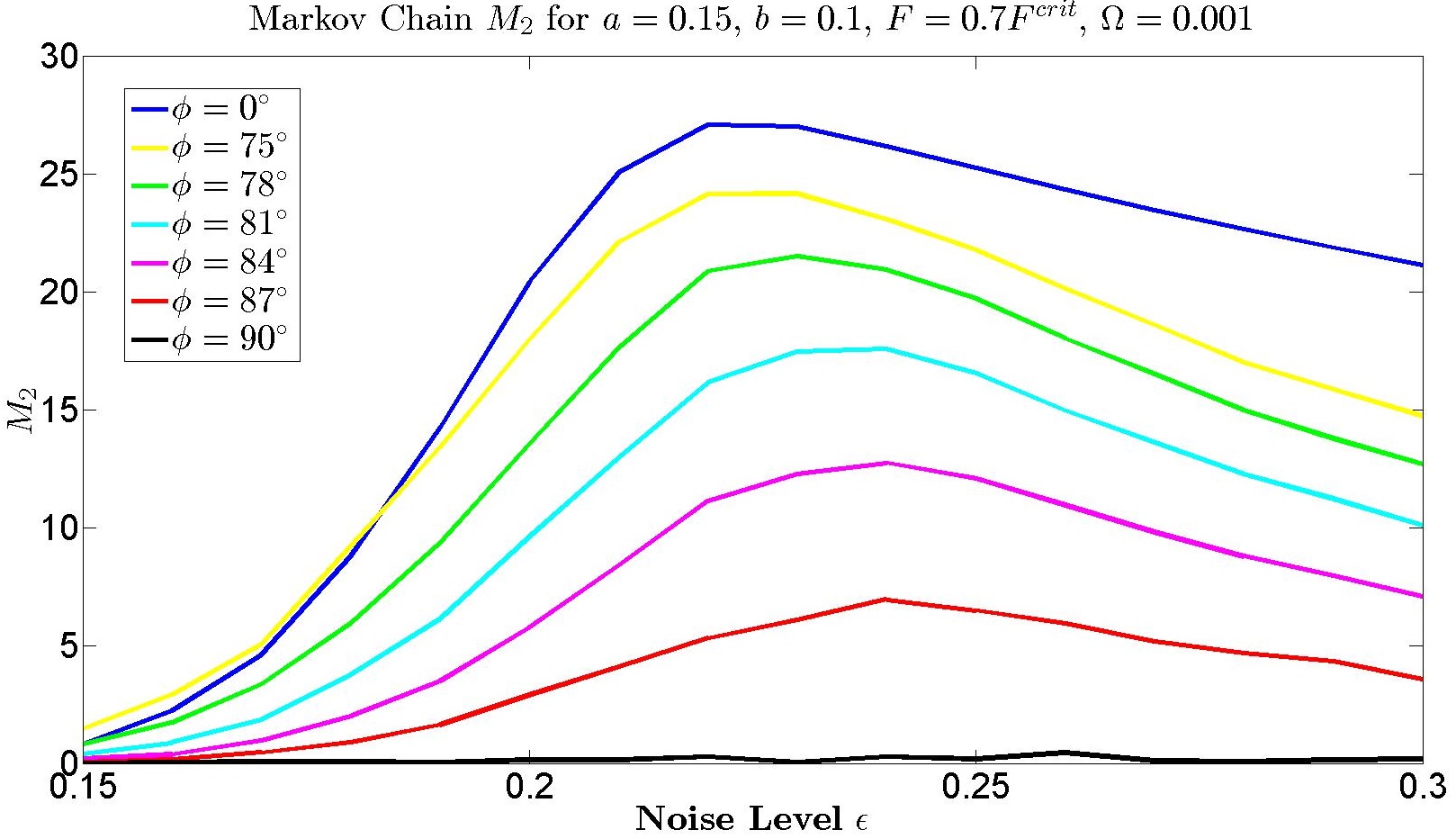}}
\caption{The measure $M_2$ for the Markov Chain for various angles and noise levels.}
\label{chap_8_g75_markov_m_2_single_measures}
\end{figure}

\begin{figure}[H]
\centerline{\includegraphics[scale=0.38]{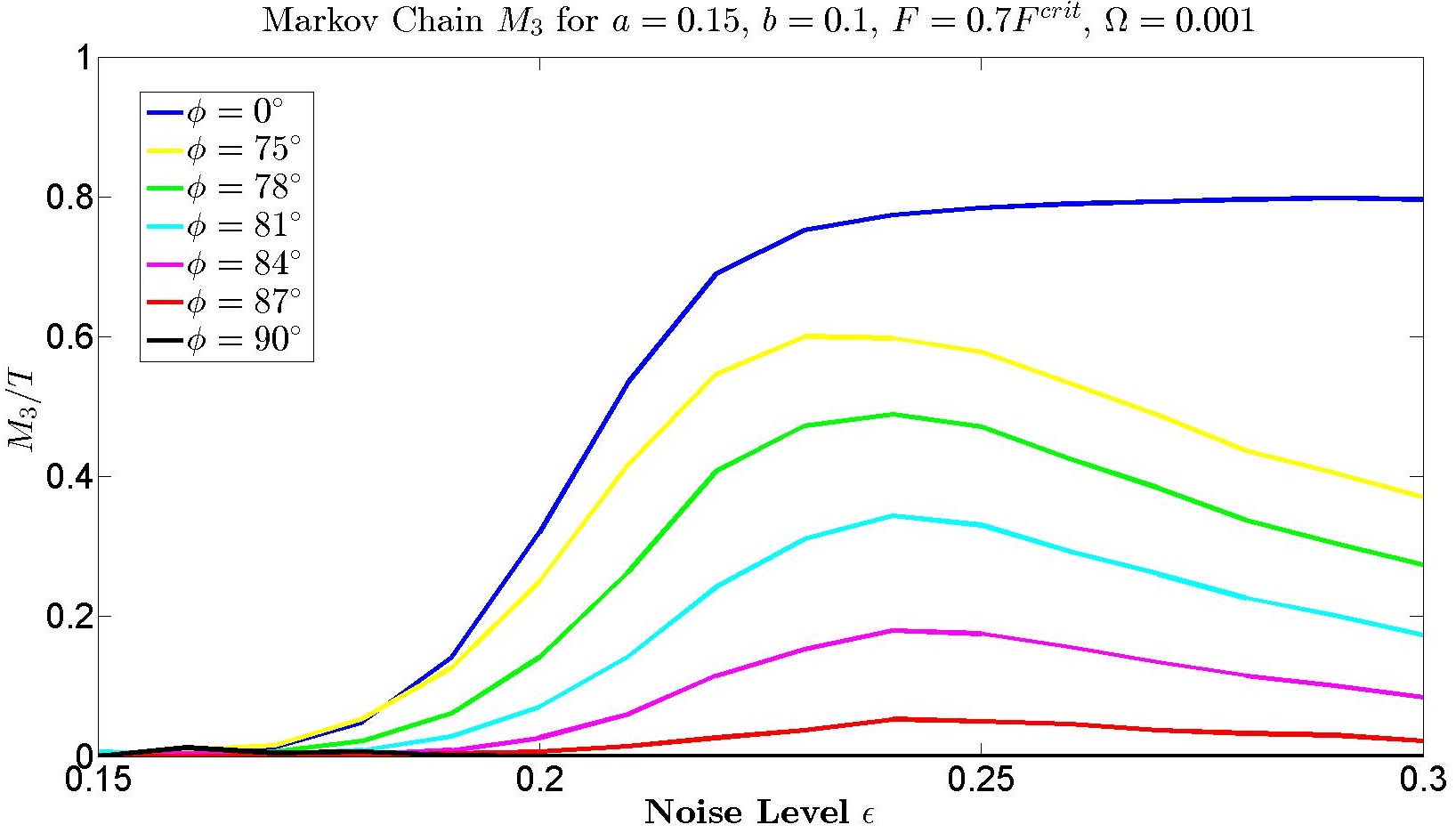}}
\caption{The measure $M_3$ for the Markov Chain for various angles and noise levels.}
\label{chap_8_g75_markov_m_3_measures}
\end{figure}

\begin{figure}[H]
\centerline{\includegraphics[scale=0.38]{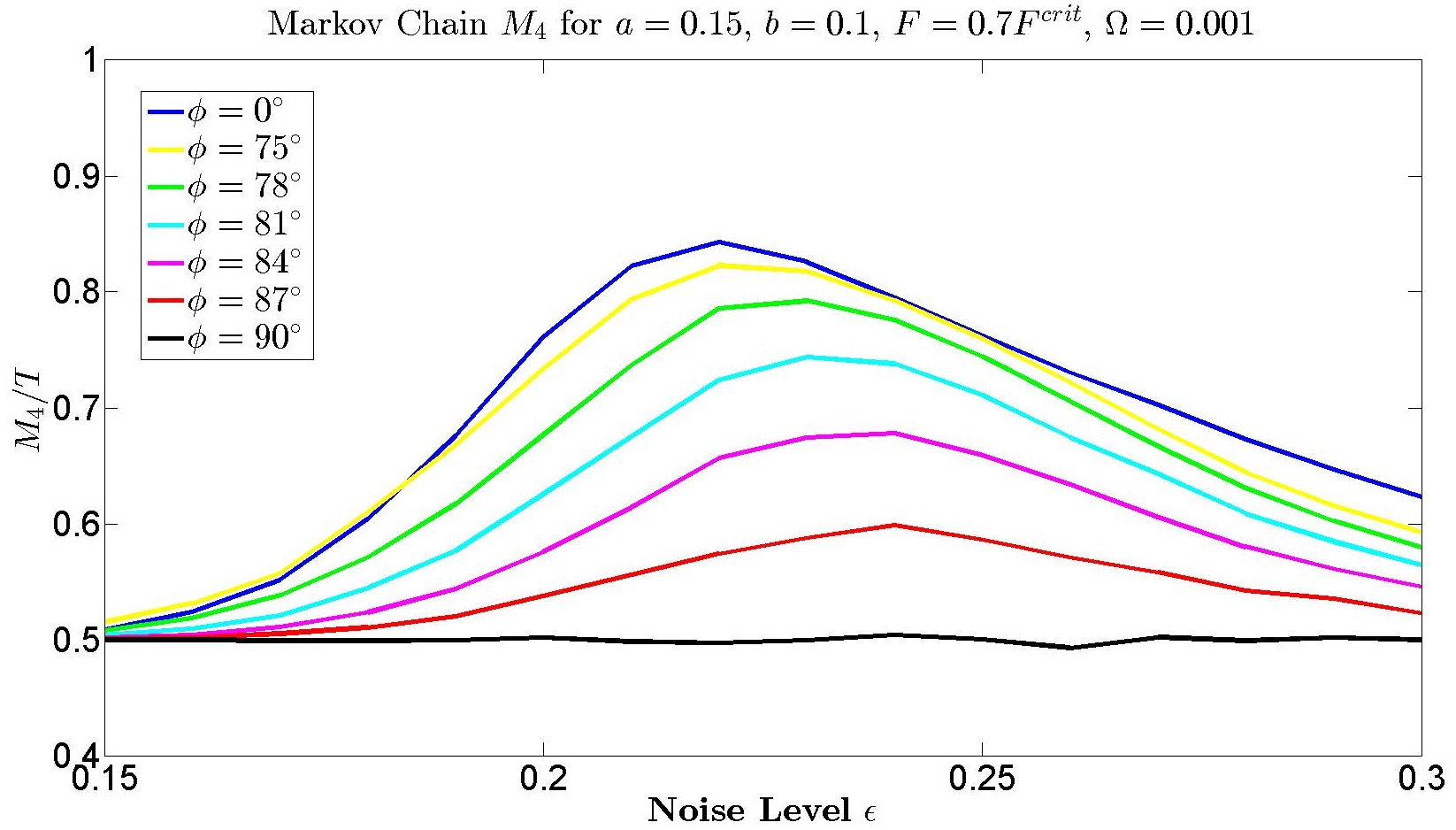}}
\caption{The measure $M_4$ for the Markov Chain for various angles and noise levels.}
\label{chap_8_g75_markov_m_4_measures}
\end{figure}

\begin{figure}[H]
\centerline{\includegraphics[scale=0.38]{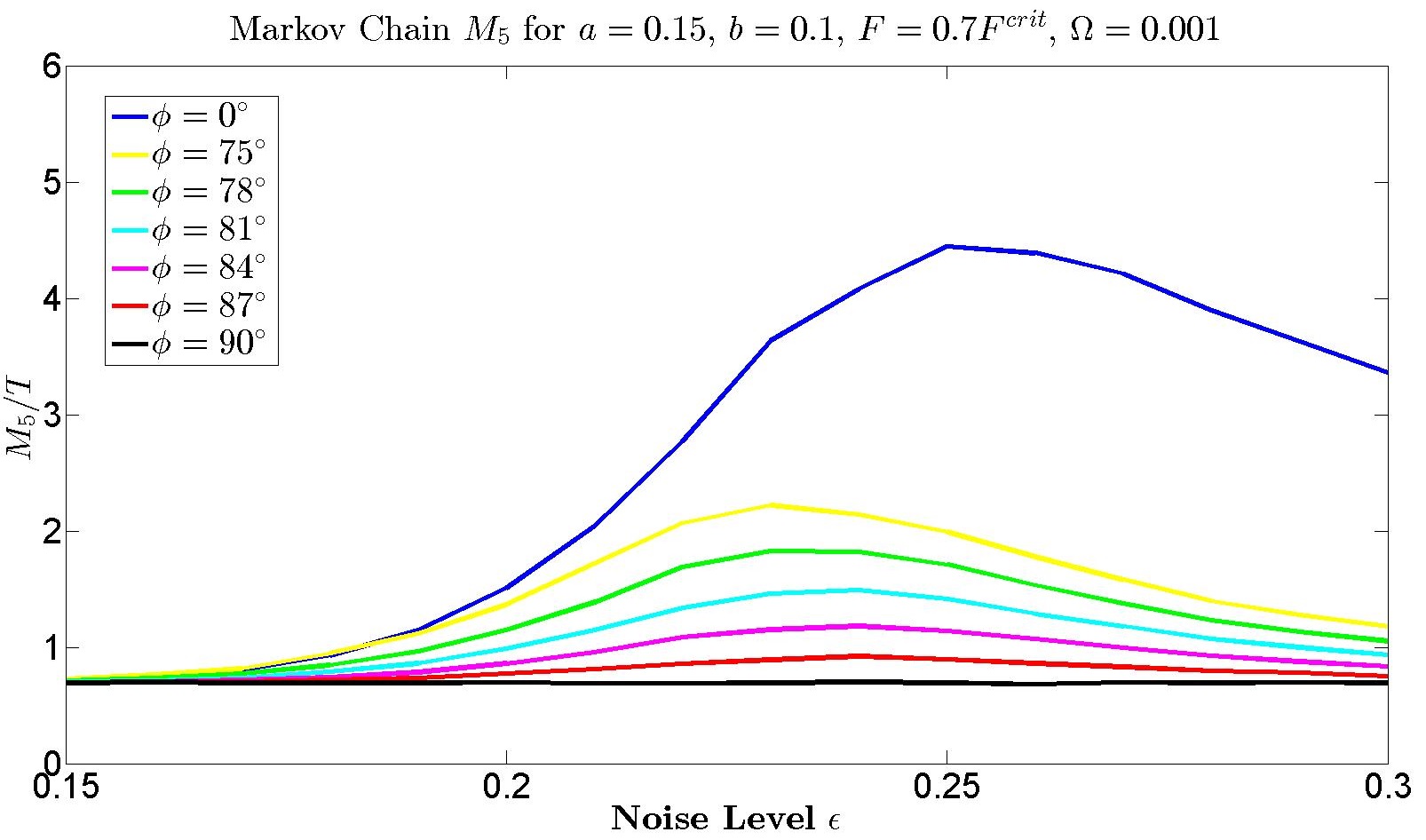}}
\caption{The measure $M_5$ for the Markov Chain for various angles and noise levels.}
\label{chap_8_g79_markov_m_5_measures}
\end{figure}

\begin{figure}[H]
\centerline{\includegraphics[scale=0.38]{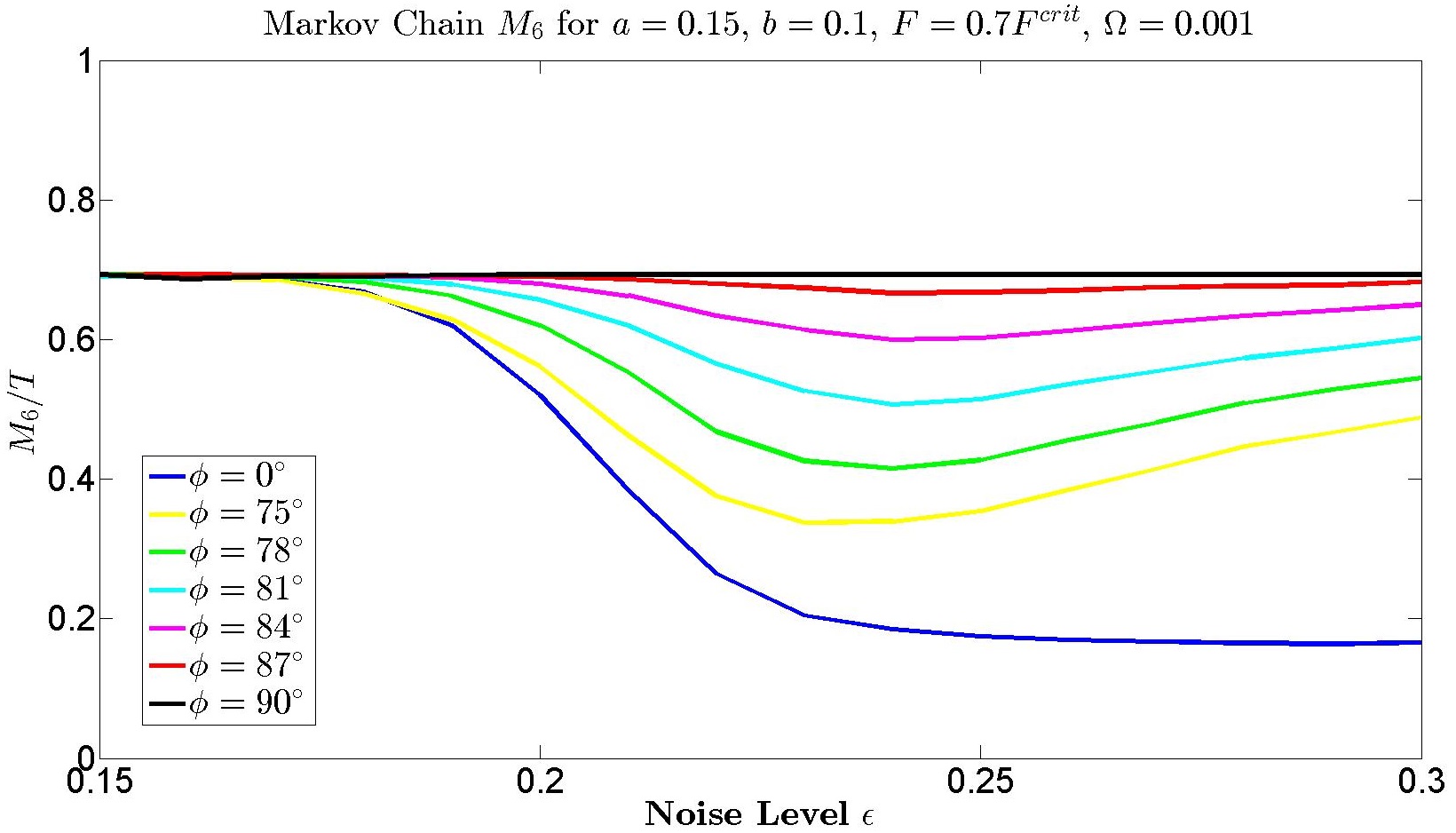}}
\caption{The measure $M_6$ for the Markov Chain for various angles and noise levels.}
\label{chap_8_g75_markov_m_6_measures}
\end{figure}

\subsubsection{Interpretation of the Six Measures Analysis}
The six measures $M_1$, $M_2$, $M_3$, $M_4$, $M_5$ and $M_6$ were plotted as a function of the noise level $\epsilon$. 
The six measures show a regular systematic behaviour in the angle $\phi$. 
The shape of the graphs of the six measures  were very similar for all the angles. 
As the angle  increased from $\phi=0^\circ$ to $\phi=90^\circ$ the six measures  gradually tended to being nearly constant in $\epsilon$. 

This effect can be explained with the invariant measures. 
When $\phi=0^\circ$ the probabilities for escaping from left to right $p_{-1+1}$ was different to to the probabilities for escaping from right to left $p_{+1-1}$.
But in the $\phi=90^\circ$ case they are the same, that is
\begin{align*}
\phi=0^\circ\,\,\, \quad p_{-1+1}&\neq p_{+1-1}\\
\phi=90^\circ \quad p_{-1+1}&= p_{+1-1}
\end{align*}
The is can be understood geometrically. 
For $\phi=0^\circ$ we have $F_x\neq0$ and $F_y=0$. 
The two wells in the Mexican Hat potential move up and down and are alternating with each other. 
When one well is high the other is low. 
For $\phi=90^\circ$ we have $F_x=0$ and $F_y\neq0$.
The two wells are always at the same height as each other and the distance to the saddles (which is a gateway to escape) is also the same in both wells. 

Recall our discussions on the Markov Chain in Section \ref{chapter_oscil_times}. 
The $p$ is related to left to right escape $p_{-1+1}$
and $q$ was related to right to left escape $p_{+1-1}$. 
For $\phi=0^\circ$ the Markov Chain can be modelled with $p\neq q$ and for $\phi=90^\circ$ the Markov Chain can be modelled with $p=q$. 
In the case of $\phi=0^\circ$ the invariant measure was cyclically changing in time. 
In the case of $\phi=90^\circ$ the invariant measure was constant at $\overline{\nu}_-(\cdot)=\overline{\nu}_+(\cdot)=\frac{1}{2}$. 
This explains why the six measures $M_1$, $M_2$, $M_3$, $M_4$, $M_5$ and $M_6$ were nearly constant for angle $\phi=90^\circ$. 
As $\phi$ changed from $\phi=0^\circ$ to $\phi=90^\circ$, the Markov Chain changed from being modelled by $p\neq q $ to being modelled by $p=q$. 
This explains the change in the six measures tending to being constant in $\epsilon$ as $\phi$ was varied. 
The six measures can be thought of as a way of measuring how far away the invariant measures are from being constant. 
If the invariant measures are constant then the six measures will also be constant.\footnote{See Appendix \ref{appendix_six_measures} for how $M_5$ and $M_6$ were numerically calculated. The ideas were not that trivial.}

For fixed $\phi$ near $\phi=90^\circ$ there is no pronounced maximum of any measure for varying $\epsilon$. 
Hence the six measures indicate the absence of a pronounce stochastic resonance near $\phi=90^\circ$. 
But consider the trajectories at a range of angles. 

\begin{figure}[H]
\centerline{\includegraphics[scale=0.35]{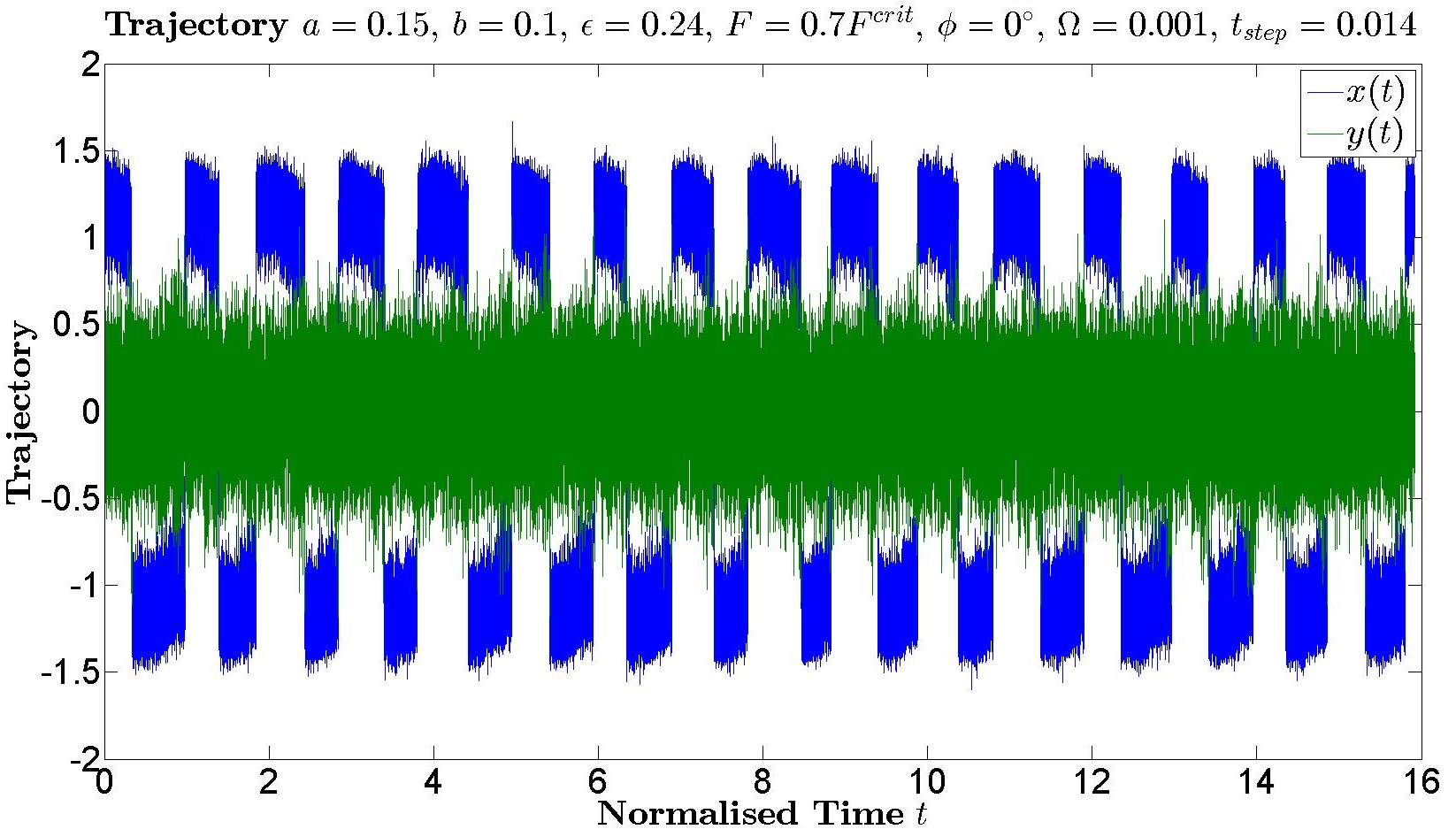}}
\caption{The blue trajectory is $x(t)$ and the green trajectory is $y(t)$.}
\label{chap_8_path_p0}
\end{figure}

\begin{figure}[H]
\centerline{\includegraphics[scale=0.35]{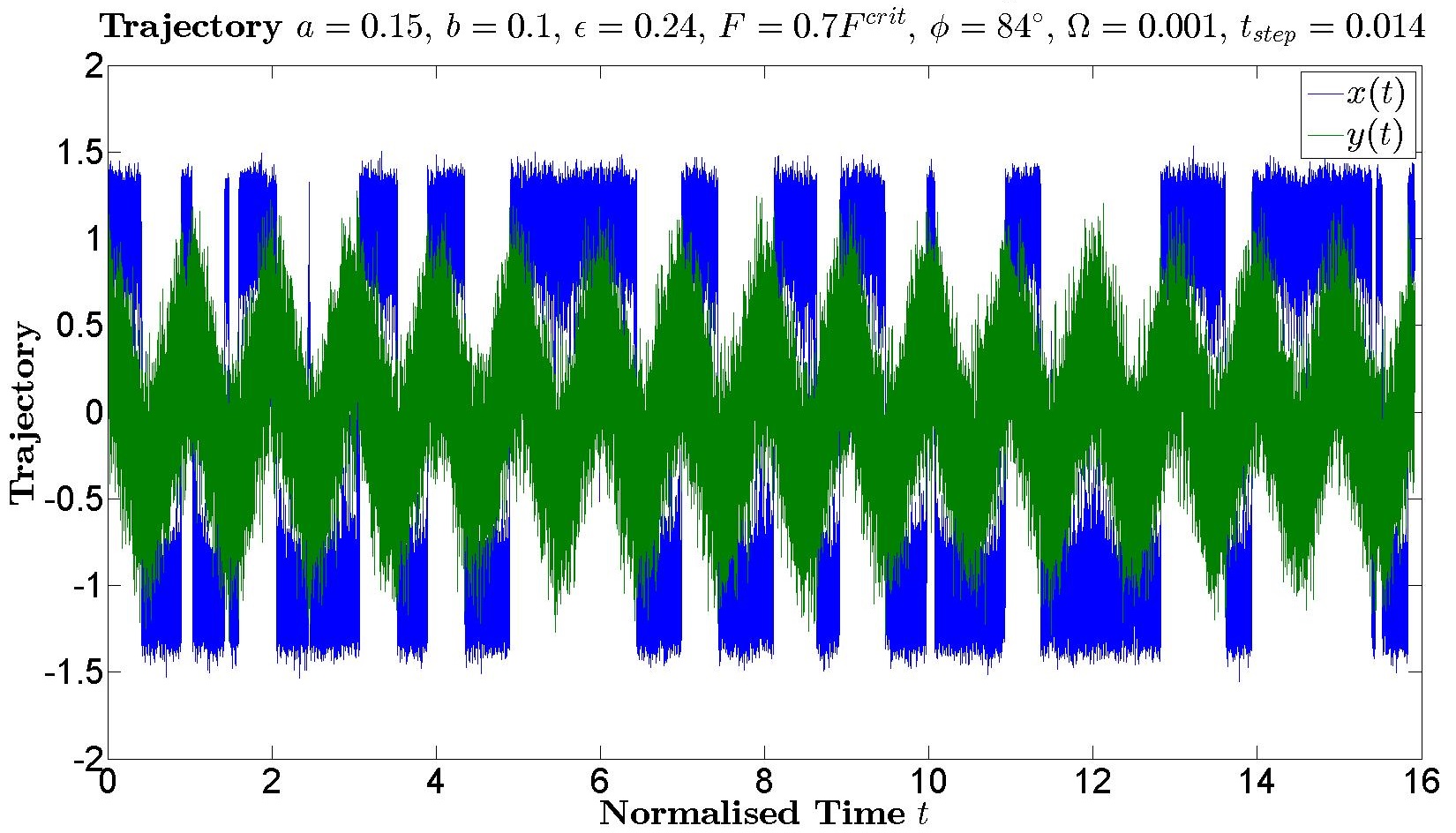}}
\caption{The blue trajectory is $x(t)$ and the green trajectory is $y(t)$.}
\label{chap_8_path_p84}
\end{figure}

\begin{figure}[H]
\centerline{\includegraphics[scale=0.35]{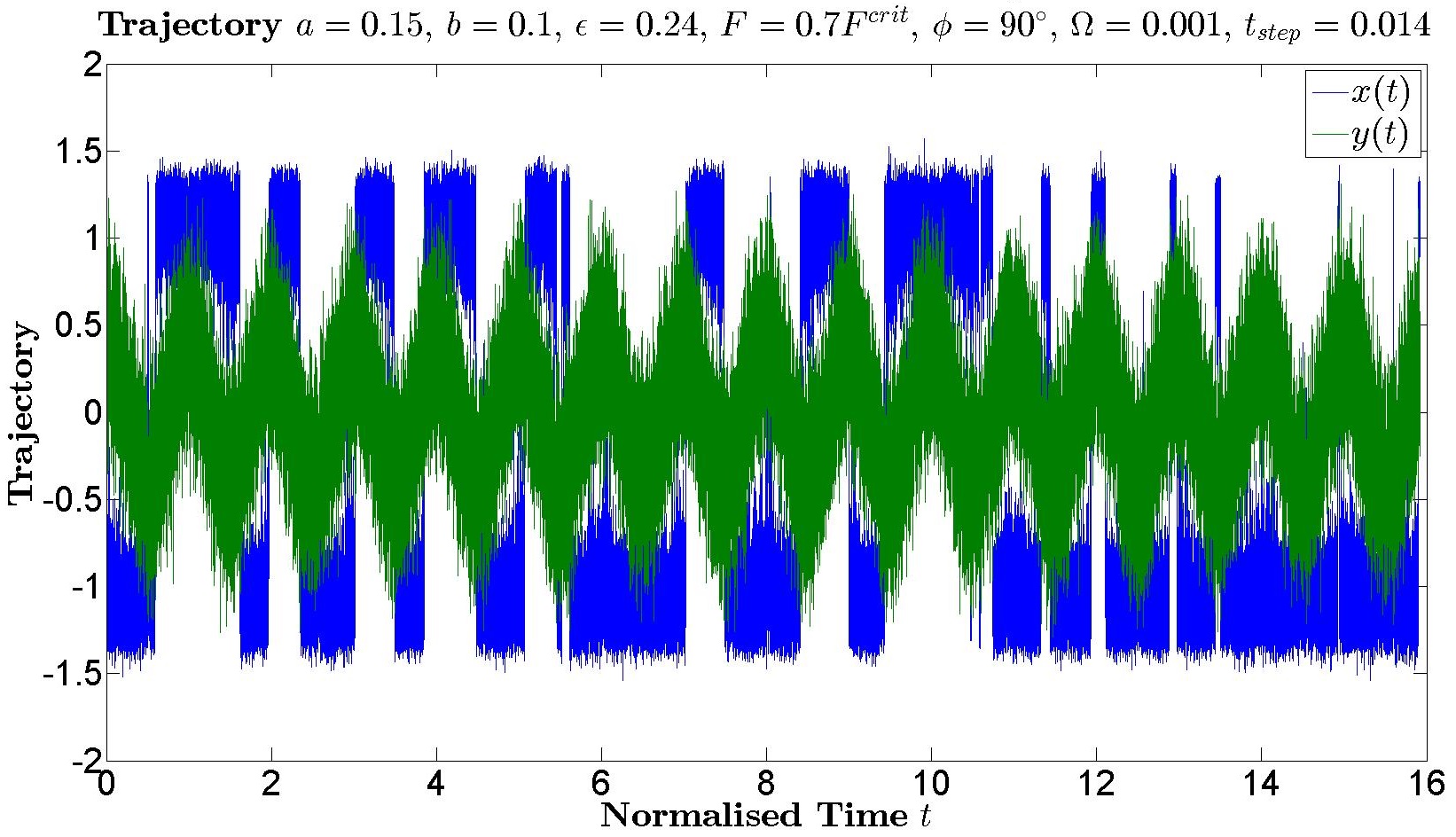}}
\caption{The blue trajectory is $x(t)$ and the green trajectory is $y(t)$.}
\label{chap_8_path_p90}
\end{figure}

\noindent When $\phi=0^\circ$ the $x(t)$ show quasi-deterministic behaviour. 
The transitions are very regular and $y(t)$ fluctuates around zero. 
As the angle varies the transitions become less regular and $y(t)$ starts to oscillate. 
This suggest that there is some regularity in the behaviour of the trajectories but the six measures are not detecting it. 
Further studies with the escape times would tell us more.

\subsection{Escape Time and Conditional KS Test Analysis}
We remind ourselves of the PDF of escape times and the way the conditional KS test can be applied in our context. 
The conditional PDF of the escape times are
\begin{align*}
p_{-}(t,u)&=R_{-1+1}(t)\exp\left\{-\int^t_uR_{-1+1}(s)\,ds\right\}\\[0.5em] 
p_{+}(t,u)&=R_{+1-1}(t)\exp\left\{-\int^t_uR_{+1-1}(s)\,ds\right\}
\end{align*} 
where $R_{-1+1}$ and $R_{+1-1}$ are the Kramers escape rate from left to right and right to left. 
In the case of $p_-(t,u)$, $t$ is the time coordinate of escape from the left well and $u$ is the time coordinate of entrance into the left well. 
In the case of $p_+(t,u)$, $t$ is the time coordinate of escape from the right well and $u$ is the time coordinate of entrance into the right well.
If we do not differentiate between escaping from the left or right then the PDF for an escape time $t$ is 
(note that $t$ here is an escape time as it is and not a time coordinate)
\begin{align*}
p_{tot}(t)=
\frac{1}{2}
\int_{0}^{T}p_-(t+u,u)m_-(u)+p_+(t+u,u)m_+(u)\,du
\end{align*}
where $m_-(\cdot)$ and $m_+(\cdot)$ are PDFs of the time of entrance into the left and right well respectively. 
We do not have  explicit expressions for  $m_-(\cdot)$ and $m_+(\cdot)$.
The  $p_{tot}(t)$ is approximated by
\begin{align*}
p_{tot}(t)\approx p_+(t,0)
\end{align*}
The times it took to escape from both the left or right wells are plotted in  histograms. 
This is an empirical approximation to the PDF $p_{tot}(t)\approx p_+(t,0)$. 
A selection of some of the  results are given below for various angles of the forcing $\phi$ and noise level $\epsilon$. 
They are examples of the Singles, Intermediate and Double Frequencies which we will explain later. 
Note that the escape times are given in units of normalised time, which is in the number of periods $T$. 

\begin{figure}[H]
\centerline{\includegraphics[scale=0.35]{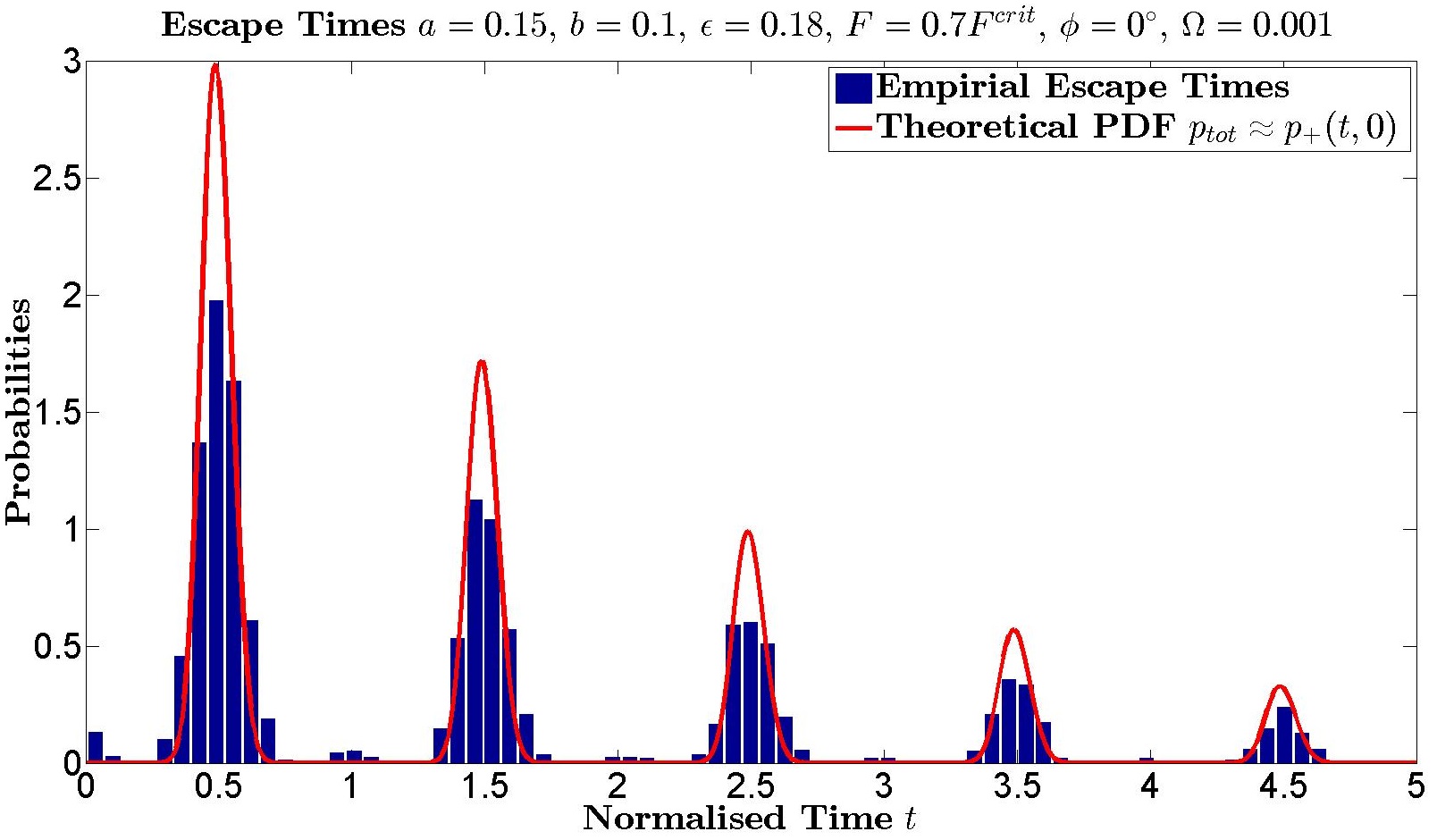}}
\caption{This is an example of the Single Frequency.
}
\label{chap_8_g75_p0_e18_pdf}
\end{figure}

\begin{figure}[H]
\centerline{\includegraphics[scale=0.35]{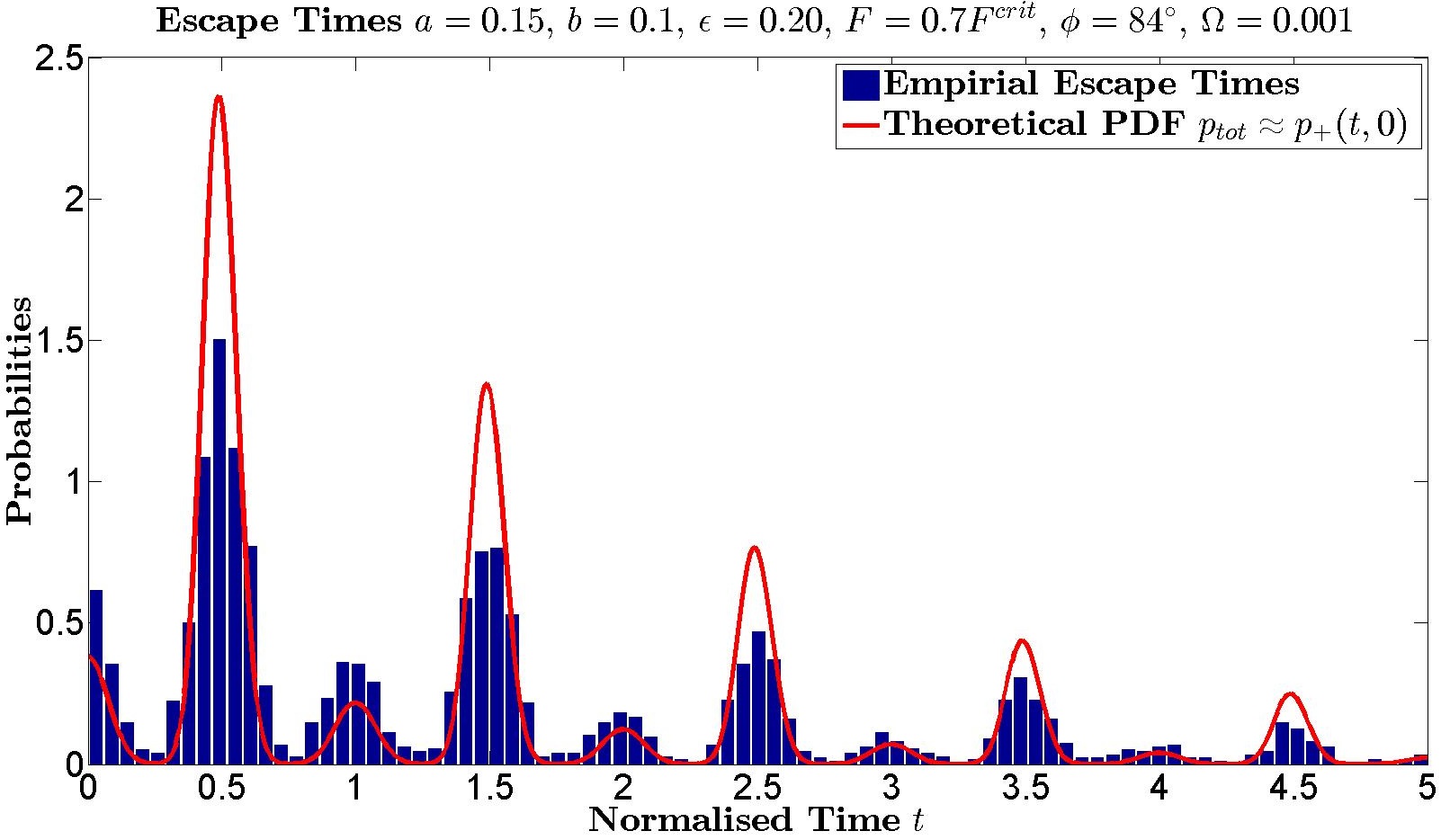}}
\caption{
This is an example of the Intermediate Frequency.
}
\label{chap_8_g75_p84_e20_pdf}
\end{figure}

\begin{figure}[H]
\centerline{\includegraphics[scale=0.35]{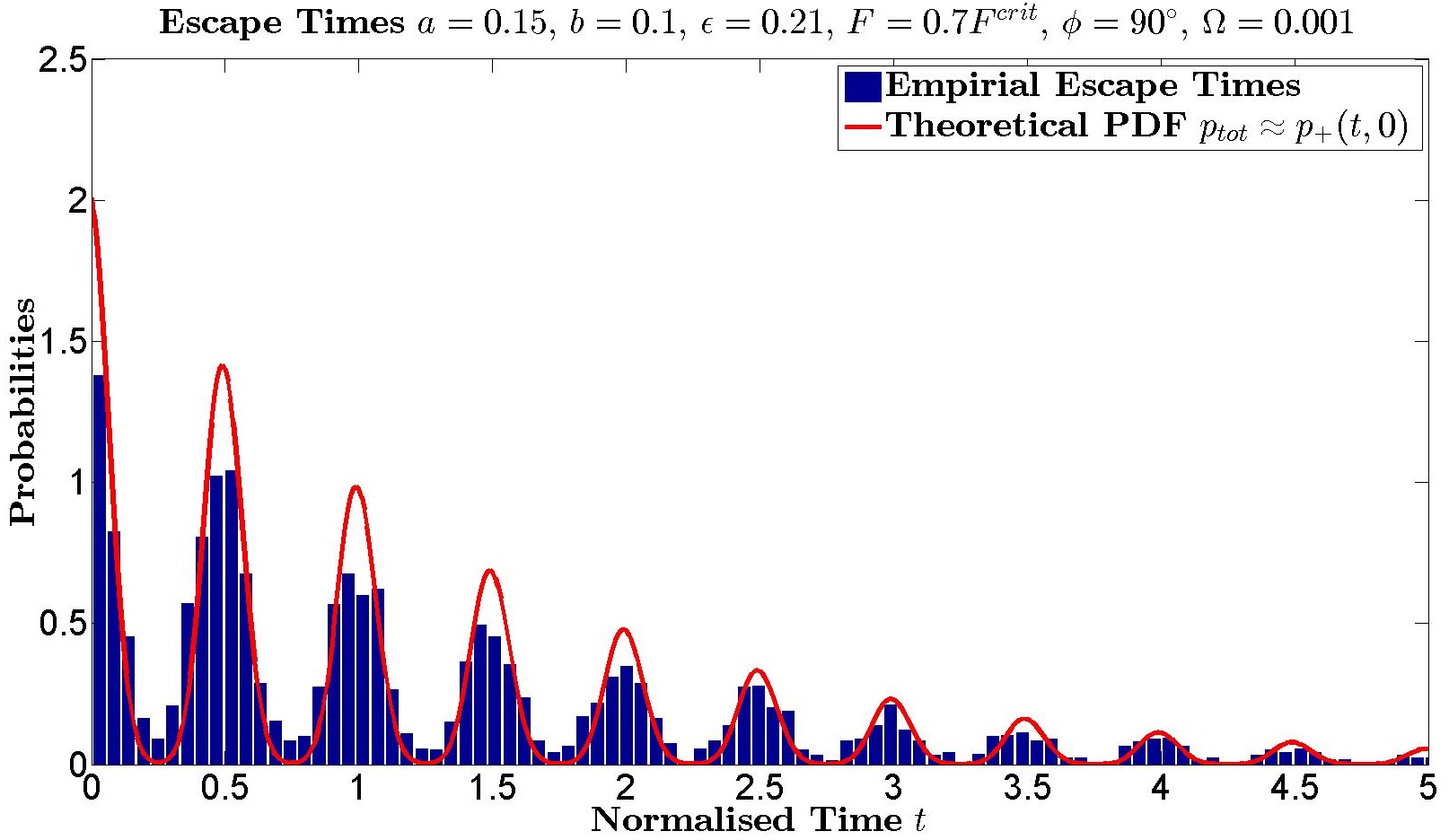}}
\caption{This is an example of the Double frequency.}
\label{chap_8_g75_p90_e21_pdf}
\end{figure}

\noindent It is important to note that Figures \ref{chap_8_g75_p0_e18_pdf}, \ref{chap_8_g75_p84_e20_pdf}
and 
\ref{chap_8_g75_p90_e21_pdf}
are histograms 
of the actual times it took to escape from either wells without differentiation between wells on the left or right. 
The times of entrance into the wells are not shown.
The PDF used is $p_{tot}(\cdot)$ which is being approximated by 
$p_{tot}(t)\approx p_+(t,0)$.

These escape times can be analysed in a different way. 
Let $u$ be the time of entrance into a well and $t$ the time of exit from a well. 
Figures \ref{chap_8_g75_p0_e18_pdf}, \ref{chap_8_g75_p84_e20_pdf} and 
\ref{chap_8_g75_p90_e21_pdf}
are therefore histograms of the $(t-u)$ for both left and right escapes combined.
Thus $0\leq mod(u,T)\leq1$ is the phase of entrance into a well and $mod(t-u,T)$ is the escape time itself in normalised time. 
Such an analysis is done for the times in Figure \ref{chap_8_g75_p0_e18_pdf}
for both the left and right wells respectively. 

\begin{figure}[H]
\centerline{\includegraphics[scale=0.35]{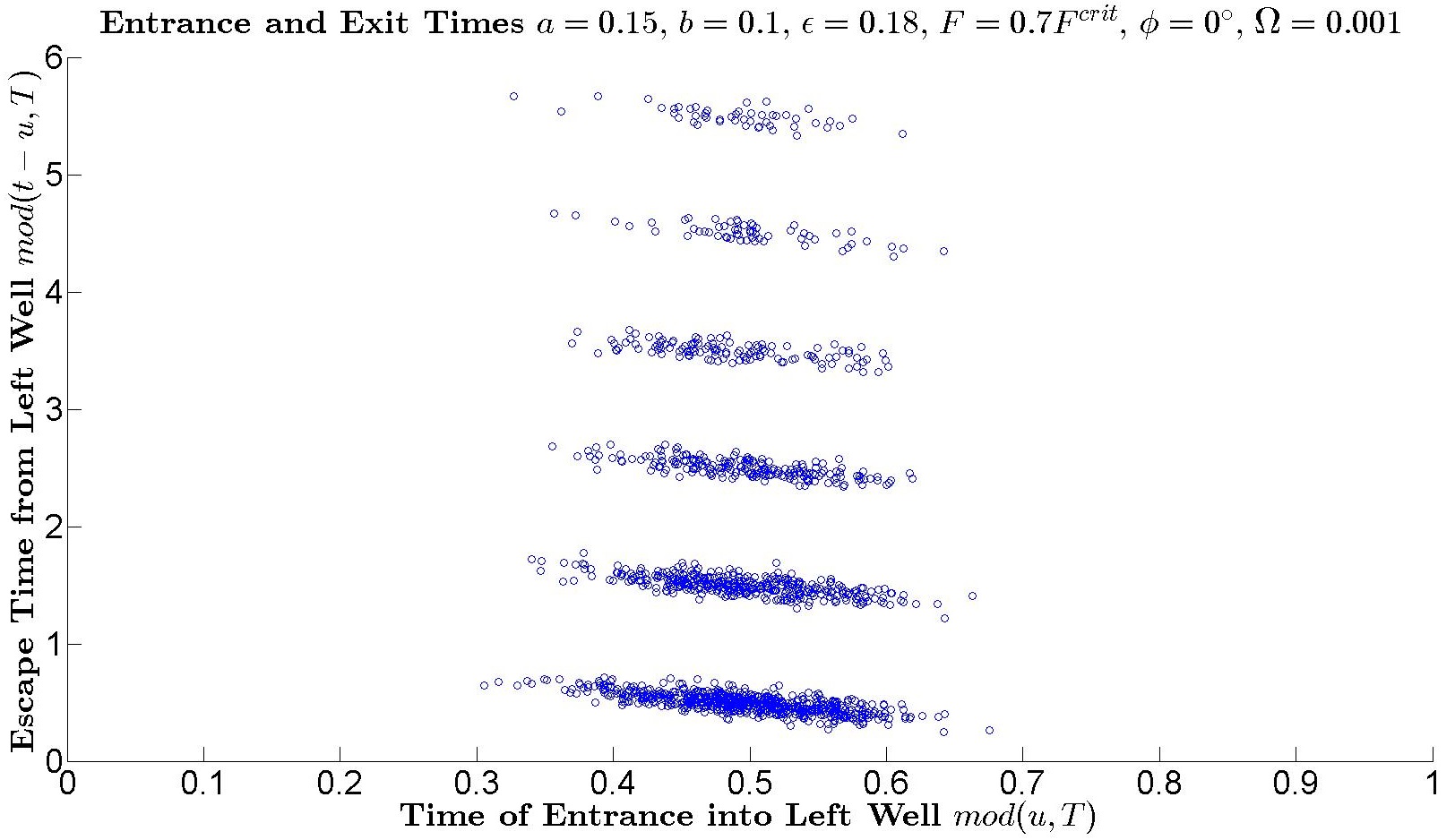}}
\caption{The $u$ is the time of entrance into the well and $t$ is the time of exit from the well.}
\label{chap_8_scatter_g75_p0_e18_pmaT}
\end{figure}

\begin{figure}[H]
\centerline{\includegraphics[scale=0.35]{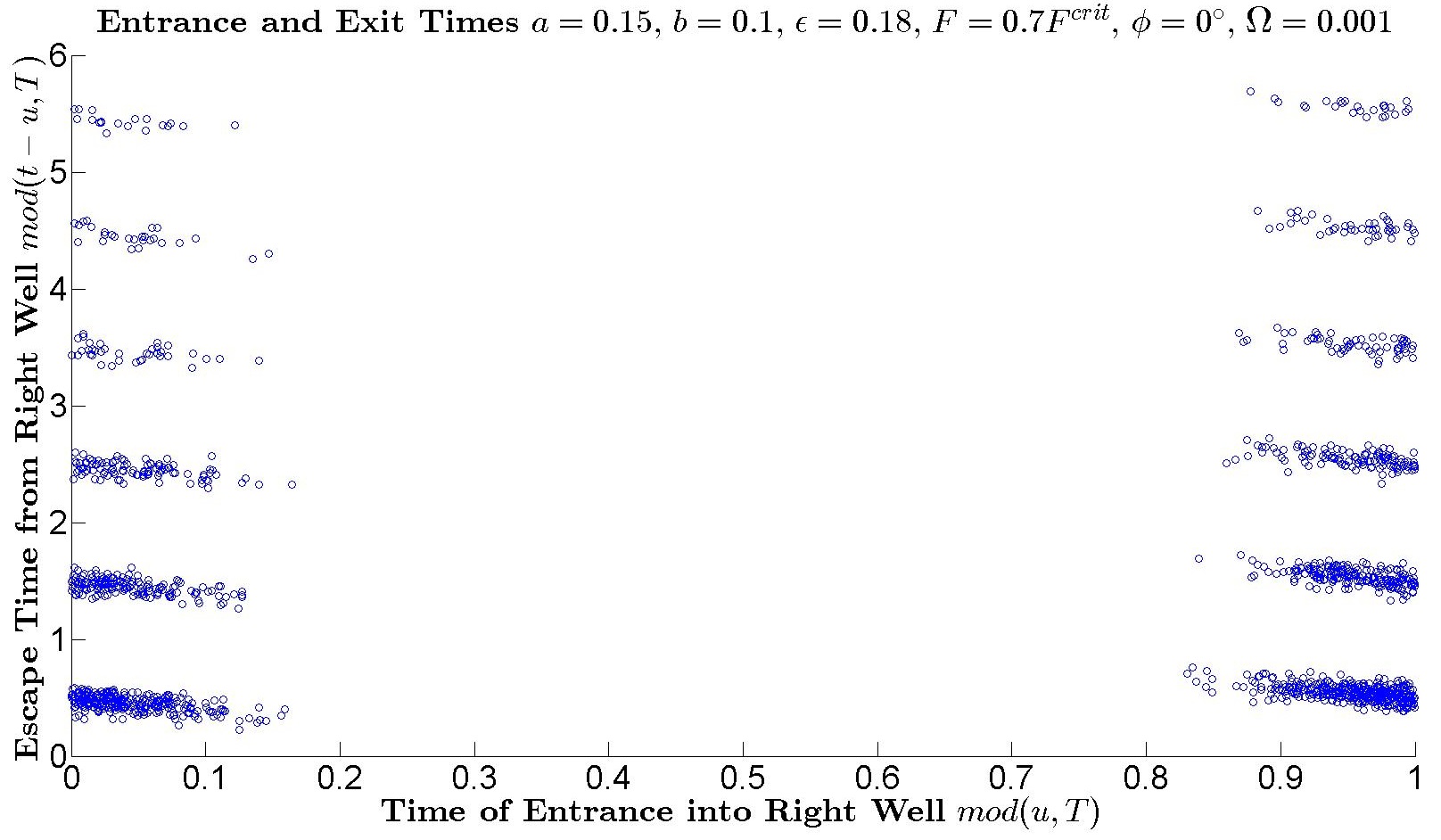}}
\caption{The $u$ is the time of entrance into the well and $t$ is the time of exit from the well.}
\label{chap_8_scatter_g75_p0_e18_ppaT}
\end{figure}

\noindent 
Notice the general behaviour of the data for $mod(u,T)$ and $mod(t-u,T)$. 
For the $\phi=0^\circ$ case the wells are alternating and one well is higher than the other. 
Entrance into the left well tend to occur near $u=0.5$ and entrance into the right well tend to occur near $u=0$ and $u=1$. 
For $\phi=90^\circ$ the wells are synchronised and are always at the same height as each other. 
Entrance and exit to and from either well tend to occur at $u=0$, $u=0.5$ and $u=1$.

Notice also  in Figure \ref{chap_8_scatter_g75_p0_e18_pmaT}
the data points are tiled near $0.5$. 
This seems to suggest that 
 the use of the Dirac delta function to approximate $p_{tot}\approx p_+(t,0)$
(see Section \ref{chap_approx_pdf}) may not be very good. 
The main problem here is the fact that 
we do not have an explicit formula for a probability measure of the time of entrance into a well, that is we do not have expressions for $m_-(u)$ and $m_+(u)$. 
This motivates us into developing the conditional KS test. 

We want to test whether the escape times we have measured are really distributed by the conditional PDFs $p_-(t,u)$ and $p_+(t,u)$. 
This is testing the conditional null hypothesis. 
Define the conditional CDFs by 
\begin{align*}
F^-_u(t)&=\int_u^tp_-(s,u)\,ds=1-\exp\left\{-\int_u^tR_{-1+1}(s)\,ds\right\}\\[0.5em]
F^+_u(t)&=\int_u^tp_-(s,u)\,ds=1-\exp\left\{-\int_u^tR_{-1+1}(s)\,ds\right\}
\end{align*}
The time coordinates of the entrance and exit from the  wells are collected. 
These are 
\begin{align*}
\left(
\begin{array}{cccc}
u_1&u_2&\ldots&u_n\\
t_1&t_2&\ldots&t_n
\end{array}
\right)
\end{align*}
where $u_i$ is the time coordinate of the $i$th entrance into a well and $t_i$ is the time coordinate of the $i$th exit from a well. 
The conditional KS statistic is calculated by 
\begin{align*}
S_n^-=\sup_{x\in[0,1]}
\left\Vert 
\frac{1}{n}
\sum_{i=1}^n\mathbf{1}_{[0,x]}
\left(
F^-_{u_i}(t_i)
-x
\right)
\right\Vert
\quad \text{and} \quad
S_n^+=\sup_{x\in[0,1]}
\left\Vert 
\frac{1}{n}
\sum_{i=1}^n\mathbf{1}_{[0,x]}
\left(
F^+_{u_i}(t_i)
-x
\right)
\right\Vert
\end{align*}
where in $S_n^-$ we sum over the time coordinates of entrance and exit to and from the left well
and in $S_n^+$ we sum over the time coordinates of entrance and exit to and from the right well. 
Recall that if the conditional null hypothesis is true then $S_n^-$ and $S_n^+$ are asymptotically distributed by 
\begin{align*}
\lim_{n\longrightarrow \infty} P(\sqrt{n} S_n \leq x)=Q(x)
\quad \text{where} \quad  
Q(x)=1-2\sum_{k=1}^\infty (-1)^{k-1}e^{-2k^2x^2}
\end{align*}
We want  99\% confidence. 
Note that 
\begin{align*}
P\left(
\sqrt{n}S_n\leq 1.6920
\right)
=Q(1.6920)
=0.99
\end{align*}
The $Q\left(\sqrt{n}S_n\right)$ is also calculated. 
The smaller $Q\left(\sqrt{n}S_n\right)$ is the more certain we are in accepting the null hypothesis. 
A selection of some of the data being implemented with the conditional KS test are given below for various angles of the forcing $\phi$ and noise level $\epsilon$. 
These are examples of the KS test being implemented for the histograms of escape times just given in Figures \ref{chap_8_g75_p0_e18_pdf} and 
\ref{chap_8_g75_p90_e21_pdf}

\begin{figure}[H]
\centerline{\includegraphics[scale=0.35]{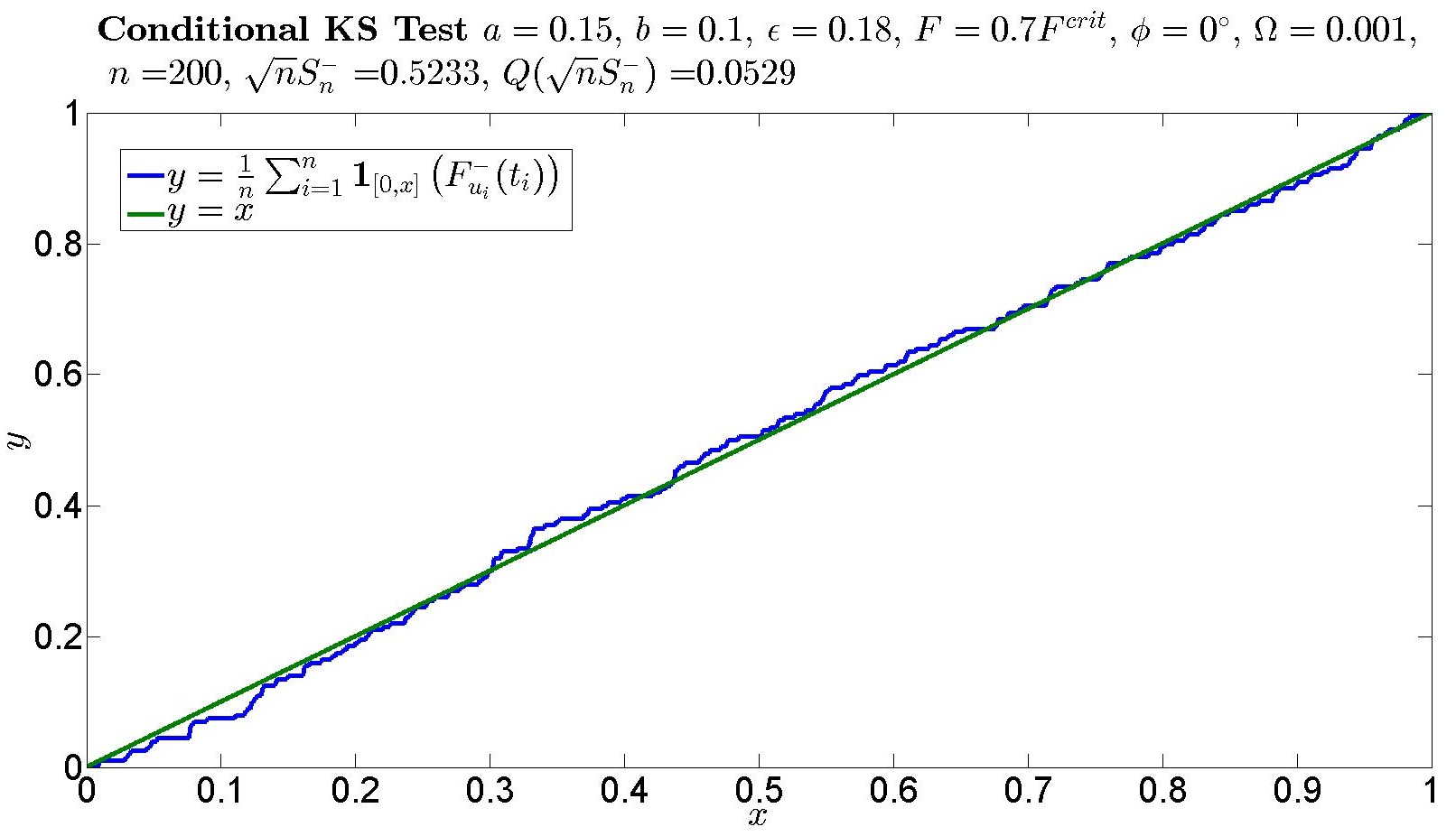}}
\caption{
This is an example of the conditional KS test being implemented for the data in Figure \ref{chap_8_g75_p0_e18_pdf}.
Note that $\epsilon=0.18$, $\phi=0^\circ$, $n=200$, $\sqrt{n}S^-_n=0.5233$ and $Q\left(\sqrt{n}S^-_n\right)=0.0529$.
}
\end{figure}

\begin{figure}[H]
\centerline{\includegraphics[scale=0.35]{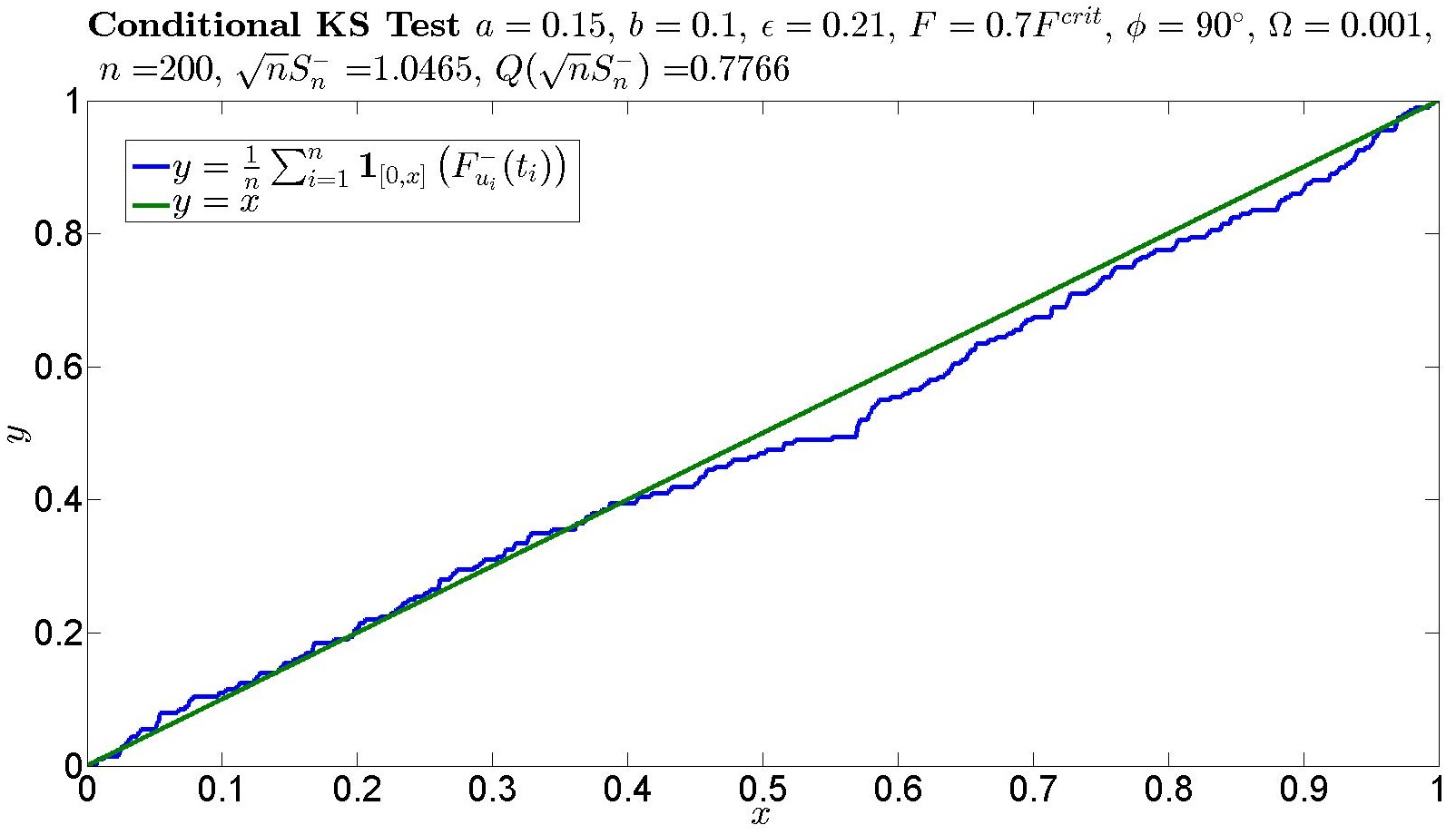}}
\caption{
This is an example of the conditional KS test being implemented for the data in Figure \ref{chap_8_g75_p90_e21_pdf}.
Note that $\epsilon=0.21$, $\phi=90^\circ$, $n=200$, $\sqrt{n}S^-_n=1.0465$ and $Q\left(\sqrt{n}S^-_n\right)=0.7766$.
}\label{chap_8_g80_p90_e21_m1}
\end{figure}

\subsubsection{Interpretation of the Escape Time and Conditional KS Test Analysis}
When $\phi=0^\circ$ there were peaks in the empirical PDF of the escape times. 
These occurred at times $\frac{1}{2}T$, $\frac{3}{2}T$, $\frac{5}{2}T$, \ldots. 
This effect we call the Single frequency. 
When $\phi=90^\circ$ the peaks occurred  at 
$\frac{1}{2}T$, $\frac{3}{2}T$, $\frac{5}{2}T$, \ldots
and $0$, $T$, $2T$, $3T$, $4T$, \ldots. 
This effect we call the Double Frequency. 
When $0^\circ<\phi<90^\circ$ an intermediate effect is seen. 
There were major peaks at 
$\frac{1}{2}T$, $\frac{3}{2}T$, $\frac{5}{2}T$, \ldots
and minor peaks at 
$0$, $T$, $2T$, $3T$, $4T$.

The behaviour of the Single, Intermediate and Double Frequencies can be explained geometrically. 
When the height between a well and a saddle is minimum, the optimal probability of escape has occurred. 
When $\phi=0^\circ$ the  frequency of the return of the optimal probability of escape is the same as the driving frequency  $\Omega$. 
This optimal probability comes back very $T$ which is once in a period. 
When $\phi=90^\circ$ the  frequency of the return of the optimal probability of escape is the double the driving frequency at $2\Omega$. 
This optimal probability comes back very $\frac{T}{2}$ which is twice in a period. 
This explains why the peaks in the Single and Double Frequencies are seen where they are.

As the angle changed from $\phi=0^\circ$ to $\phi=90^\circ$ the Single Frequency gradually changes into the Double Frequency with the Intermediate Frequency seen in between. 
Thus  the angle of the forcing is leaving a mark in the PDFs of escape times.

When the conditional KS test was implemented, the functions
\begin{align*}
y_0(x)=x, 
\quad 
y_-(x)=\sum_{i=1}^n\mathbf{1}_{[0,x]}
\left(
F^-_{u_i}(t_i)
\right)
\quad \text{and} \quad 
y_+(x)=\sum_{i=1}^n\mathbf{1}_{[0,x]}
\left(
F^+_{u_i}(t_i)
\right)
\end{align*}
were used to calculate the following distances which are the conditional KS statistics
\begin{align*}
S^-_n=\left\Vert y_0-y_-\right\Vert_\infty
\quad \text{and} \quad 
S^+_n=\left\Vert y_0-y_+\right\Vert_\infty
\end{align*}
It is reasonable to say that $y_-(\cdot)$ and $y_+(\cdot)$ were close enough to $y_0(\cdot)$ that we can accept the conditional null hypothesis. 
This can be seen and judged graphically with $S^-_n$ and $S^+_n$ calculated as well. 
This is an example of the conditional KS test giving a reasonable result.

\subsection{Remarks on Analysis of Stochastic Resonance}
There are a few subtleties, heavily dependent approximations and setbacks to the analysis which is worth mentioning here. 

\subsubsection{Remarks on Implementing the Conditional KS Test}
Notice that all the theories developed about the KS Test were based on the assumption that the null hypothesis is true. 
This means strictly speaking a small KS statistic, that is a small $S_n^-$ or $S_n^+$ does not immediately allow us to accept the null hypothesis but good reasons not to reject it. 
Also when there were many transitions, that is for large $n$, the terms $Q(\sqrt{n}S_n^-)$ and $Q(\sqrt{n}S_n^+)$ were also calculated.
The smaller  $Q(\sqrt{n}S_n^-)$ and $Q(\sqrt{n}S_n^+)$ are the more confidence we have in not rejecting the null hypothesis. 
This is because 
for very large $n$, we would expect
\begin{align*}
\lim_{n\rightarrow\infty}\sqrt{n}S_n^-=0
\quad \text{and} \quad 
\lim_{n\rightarrow\infty}\sqrt{n}S_n^+=0
\end{align*}
so the smaller $Q(\sqrt{n}S_n^-)$ and $Q(\sqrt{n}S_n^+)$ are the more certain we are in not rejecting the null hypothesis. 
Note that we have used $n=200$ transitions for implementing the KS test. This still works with sparse data. For examples with few data say $n=20$ see the thesis \cite{tommy_thesis}. 

\subsubsection{Remarks on Adiabatic Approximation}
Notice that in the PDFs $p_-(t,u)$, $p_+(t,u)$ and $p_{tot}(t)$ expressions for the escape rates $R_{-1+1}(t)$ and $R_{+1-1}(t)$ were required. 
These rates were also required for the conditional KS test. 
Strictly speaking these rates are dependent on the driving frequency $\Omega$, 
but we stress that these rates were calculated using Kramers formula as though the particle is escaping from a static potential. 
This is the adiabatic approximation where an oscillatory potential is approximated by a static potential.

It is worth summarising all the approximations which the analysis of the data have been based.
There is the small noise approximation and slow forcing approximation from Kramers formula, the adiabatic approximation and the perfect phase approximation where $p_{tot}$ is approximated by  $p_{tot}\approx p_{tot}(t,0)$.

\section{Conclusion}

In this paper we have considered the following problem. 
Let $X^\epsilon_t$ be a stochastic process in $\mathbb{R}^2$ which is described by the the SDE 
\begin{align*}
dX^\epsilon_t=b\left(X^\epsilon_t,t\right)dt+\epsilon\,dW_t
\end{align*}
and the drift term $b(\cdot,\cdot)$ is expressed by  
\begin{align*}
b(x,t)=-\nabla V_0 (x) + F\cos \Omega t 
\end{align*}
where $V_0:\mathbb{R}^2\longrightarrow \mathbb{R}$ is a time independent function, the unperturbed potential, with two metastable states, and two pathways between these states. The $F\in\mathbb{R}^r$ is the magnitude of the forcing and $\Omega$ is the driving frequency. 
Our aim was to see characteristics of the trajectory $X^\epsilon_t$ which only depends on the qualitative structure of $V_0$, that is the existence of two metastable states and two pathways. 

For concreteness we considered a model, which we call the  Mexican Hat Toy Model
\begin{align*}
V_0(x,y)=\frac{1}{4}r^4-\frac{1}{2}r^2-ax^2+by^2
\quad \text{where} \quad r=\sqrt{x^2+y^2}
\end{align*}
The magnitude and angle of the forcing are given by
\begin{align*}
F=\sqrt{F_x^2+F_y^2}
\quad \text{and} \quad 
\phi=\tan^{-1}\left(\frac{F_y}{F_x}\right)
\end{align*}
The angle $\phi$ and noise level $\epsilon$ were varied. 
At $\phi=0$ the wells were alternating, that is one well is higher than the other, in the sense that it is easer to jump from one well to the other than vice versa.
At $\phi=90^\circ$ the wells are synchronised, that is both wells are always at the same height but the heights of the barrier for the two paths is alternating.

A potential with two pathways has never been considered before in the context of stochastic resonance. 
We studied it using approximation techniques and direct simulations. 
In an adiabatic regime the Freidlin-Wentzell theory allows one to give analytical solutions of the jump type distributions asymptotically in this regime. 
This theory predicted the appearance of additional resonance peaks at half the frequency when the angle approaches $\phi=90^\circ$.

We simulated $X_t^\epsilon$ for different values of $\phi$ and $\epsilon$
and computed for the values of angle increasing from  $\phi=0$ to $\phi=90^\circ$  the six measures $M_1$, $M_2$, $M_3$, $M_4$, $M_5$ and $M_6$  as function of the noise level. 
The first major surprise was that the graphs showed less and less pronounced minima (or maxima) and hence suggests that the phenomena of stochastic resonance gets less and less pronounced, see Section \ref{conclusion_six_measures}.
The effect of resonance seems to disappear overall.

However, considering the path $X_t^\epsilon$ itself, one sees  that there may be nevertheless some synchronisation, see Figure        \ref{chap_8_path_p0},
\ref{chap_8_path_p84} and 
\ref{chap_8_path_p90}.
To properly quantify synchronisation we considered the histograms of the escape times, which to our knowledge has not been  considered thoroughly before. 
The histograms showed a clear periodicity and also the emergence of peaks at the Double Frequency for increasing angle. 
For a quantitative consideration we assume that the entrance time is in perfect phase 
(this is when $m_-(u)$ and $m_+(u)$ \emph{can} be approximated by Dirac delta functions). 
This gives for several cases good quantitative and in general good qualitative agreement with the combined adiabatic and small noise approximation.  
Summarizing, the theoretical and the simulation results are in very good agreement.
We want to stress that in the comparison no free parameters were present and so no fitting took place. 

The fact that the six measures are blind can be explained using Markov chain models approximating the SDE. As one expects from large deviation theory, for small noise and in an adiabatic regime the SDE can be approximated by a continuous time Markov chain.
In this Markov chain model we showed that the invariant measures are constant when $\phi=90^\circ$. 
Hence we expect that the invariant measure gives in the diffusion case equal weights to the left and the right well. 
Together, this gives us the following qualitative picture of the dynamics for any angle. 
At a fixed time the probability that one sees a jump from the left to the right well or vice versa has the same probability. 
However, conditioned on the phase and the direction of the last jump, for concreteness assume that it was at phase $u$ and from the left to the right
(that is to say the particle entered the well at time $u$)
the next jump will be at  phase which is near to a multiple of $T/2$
(that is to say the particle will leave the well near the times $t=nT/2$ where $n$ is an integer). 
The jump rates will be given by the height of the potential barriers.

At $\phi=90^\circ$, the path $X_t^\varepsilon$ and $-X_t^\varepsilon$ will appear with the same probability if one starts in the invariant measure. 
This explains why the six measures are all insensitive in this case. 
The equilibration happens because the process will skip some of the jump opportunities and in this way the left-right synchronization will get lost quickly.

This new phenomena we discovered has added an additional motivation to the observation of Hermann, Imkeller, Pavlyukevich, Berglund and Gentz that the appropriate 
consideration has to be on the path level. 
Averaged quantities like the six measures can be very misleading and masking the real behaviour of the system. 
The escape time distribution shows a clear signal of stochastic resonance in accordance with the theoretical consideration. 
The presence of a two pathways manifests itself in an appearance of peaks at the Double Frequency. 
We showed that adiabatic small noise approximation gives a good statistical model.
We demonstrated that this appearance can be detected also when only a limited number of transitions is available. 
Our analysis provides us with a clear footprint indicating the existence of a second pathway. 
The angle dependence of our result should also allow us to predict the orientation of the saddles with respect to the wells.

\appendix

\section{Numerical Methods for calculating $M_5$ and $M_6$}
\label{appendix_six_measures}
Here we present how we computed $M_5$ and $M_6$ numerically.
This is how $M_5$ and $M_6$ are calculated in theory 
\begin{align*}
M_5&=\int_0^T
\phi^-(t)\ln\left(\frac{\phi^-(t)}{\overline{\nu}_-(t)}\right)+
\phi^+(t)\ln\left(\frac{\phi^+(t)}{\overline{\nu}_+(t)}\right)
dt\\
M_6&=\int^T_0
-\overline{\nu}_-(t)\ln\overline{\nu}_-(t)
-\overline{\nu}_+(t)\ln\overline{\nu}_+(t)\,
dt
\end{align*}
where 
\begin{align*}
\phi^-(t)&=
\left\{
\begin{array}{c}
1 \quad \text{if} \quad mod(t,T)\leq T/2\\
0 \quad \text{if} \quad mod(t,T)> T/2
\end{array}
\right.\\[0.5em]
\phi^+(t)&=
\left\{
\begin{array}{c}
0 \quad \text{if} \quad mod(t,T)\leq T/2\\
1 \quad \text{if} \quad mod(t,T)> T/2
\end{array}
\right.
\end{align*}
When the invariant measures are generated numerically they are finite discrete objects described by 
\begin{align*}
\nu_-&=\left\{\nu^-_1, \nu^-_2, \cdots, \nu^-_N\right\}\\
\nu_+&=\left\{\nu^+_1, \nu^+_2, \cdots, \nu^+_N\right\}
\end{align*}
The real invariant measure were close to zero sometimes and in the numerical approximation they became actually zero or even negative which lead to numerical artefacts.
Note that 
\begin{align*}
\lim_{x\longrightarrow0}\ln\left(\frac{1}{x}\right)=\infty
\quad \text{and} \quad 
\lim_{x\longrightarrow0}x\ln\left(x\right)=0
\end{align*}
Define
\begin{align*}
\nu_-^{lim}&=\min_{\substack{i=1,2,\ldots,N\\\nu^-_i>0}}\left\{\nu^-_1, \nu^-_2, \cdots, \nu^-_N\right\}\\
\nu_+^{lim}&=\min_{\substack{i=1,2,\ldots,N\\\nu^+_i>0}}\left\{\nu^+_1, \nu^+_2, \cdots, \nu^+_N\right\}
\end{align*}
The quantities $M_5$ and $M_6$ are computed numerically in the following way  
\begin{align*}
M_5&=\sum_{\substack{i\leq\frac{N}{2}\\\nu^-_i>0}}t_{step}\ln\left(\frac{1}{\nu^-_i}\right)
+
\sum_{\substack{i\leq\frac{N}{2}\\\nu^-_i\leq0}}t_{step}\ln\left(\frac{1}{\nu_-^{lim}}\right)
+
\sum_{\substack{i>\frac{N}{2}\\\nu^+_i>0}}t_{step}\ln\left(\frac{1}{\nu^+_i}\right)
+
\sum_{\substack{i>\frac{N}{2}\\\nu^+_i\leq0}}t_{step}\ln\left(\frac{1}{\nu^{lim}_+}\right)\\[0.5em]
M_6&=\sum_{\substack{i=1,2,\cdots,N\\\nu^-_i>0}}\nu^-_i\ln(\nu^-_i)(-t_{step})\quad+\quad
\sum_{\substack{i=1,2,\cdots,N\\\nu^+_i>0}}\nu^+_i\ln(\nu^+_i)(-t_{step})
\end{align*}

\addcontentsline{toc}{section}{References}
\bibliographystyle{ieeetr}
\bibliography{paper_reportreferences}

\end{document}